\documentclass[11pt]{article}
\usepackage[utf8]{inputenc}
\usepackage{amsfonts,amsthm,amsmath,fourier}
\usepackage[a4paper,scale={.8,.8}]{geometry}

\usepackage[usenames, dvipsnames]{xcolor}
\usepackage[colorlinks=true, pdfstartview=FitV,linkcolor=ForestGreen,citecolor=ForestGreen, urlcolor=blue]{hyperref}
\usepackage[integrals]{wasysym}
\usepackage{enumerate}
\usepackage{fancyhdr}
\usepackage{amssymb}
\usepackage{chemarrow}
\usepackage{tikz}
\usepackage{mathtools}
\usepackage{enumitem}
\usepackage{mathrsfs}

\numberwithin{equation}{section}
\newtheorem{theorem}{Theorem}[section]
\newtheorem{lemma}[theorem]{Lemma}
\newtheorem{definition}[theorem]{Definition}
\newtheorem{proposition}[theorem]{Proposition}
\newtheorem{corollary}[theorem]{Corollary}
\theoremstyle{remark}
\newtheorem{remark}{Remark}
\allowdisplaybreaks

\def\div{ \hbox{\rm div}\,  }

\def\N{{\mathbb N}}
\def\R{{\mathbb R}}
\def\T{{\mathbb T}}

\begin{document}
\title
{Global well-posedness and long-time asymptotics\\of a general nonlinear  non-local Burgers Equation}

\author{Jin Tan\footnotemark \quad\&\quad Francois Vigneron\footnotemark}

\maketitle

\begin{abstract}
This paper is concerned with the study of a nonlinear non-local equation that has a commutator structure.
The equation reads
$$\partial_t u-F(u) \, (-\Delta)^{s/{2}} u+(-\Delta)^{s/{2}} (uF(u))=0, \quad x\in  \mathbb{T}^d,$$ 
with $s\in(0, 1]$. 
We are interested in solutions stemming from periodic \textit{positive} bounded initial data.
The given  function $F\in \mathcal{C}^{\infty}(\mathbb{R}^+)$ must satisfy
$F'>0$  {\rm{a.e.}} on~$(0, +\infty)$. For instance, all the functions $F(u)=u^{n}$ with $n\in\mathbb{N}^\ast$ are admissible
non-linearities.

We construct global classical solutions starting from smooth positive data, and global weak solutions starting from positive data in $L^\infty.$ We show that any weak solution is instantaneously regularized into $\mathcal{C}^\infty.$ We also describe the long-time asymptotics of all solutions. Our methods follow several recent advances in the regularity theory of parabolic integro-differential equations,
in particular \cite{Im1, Im2}.

\end{abstract}

\maketitle
\section{\bf Introduction}
 In the book \cite{Lemarie2002},  P.G. Lemari\'{e}-Rieusset proposed the following model
 \begin{equation}\label{sclarNS}
\partial_t u+u|\nabla| u-|\nabla| (u^2)=\nu\Delta u, \quad x\in\mathbb{R}^d ~{\rm{or}} ~ \mathbb{T}^d
\end{equation}
 as an active scalar (\textit{i.e.} $u\in\R$) case study of the 3D Navier-Stokes equations, where $|\nabla|=(-\Delta)^{1/2}$ denotes the square root of the
 Laplacian, \textit{i.e.} the Fourier multiplier of symbol $|\xi|.$
 The works of F. Leli\`{e}vre \cite{Lelievre1, Lelievre2, Lelievre3} presented the construction of global Kato-type mild solutions for initial data in $L^3(\mathbb{R}^3)$ and of global weak Leray-Hopf type solutions for initial data in $L^2(\mathbb{R}^3)$ and other similar spaces. A local energy inequality obtained for this model was suggestive of possible uniqueness for small initial data in critical spaces, in a similar fashion to the 3D Navier-Stokes equations, as stated in \textit{e.g.}~\cite{Ba11}.

Recently, the works of C.~Imbert, T.~Jin, R.~Shvydkoy and F.~Vigneron \cite{Im1, Im2} have focused on the model without viscosity (note the opposite signs)
on $\T^d$:
\begin{equation}\label{NB}
\partial_t u-u|\nabla| u+|\nabla| (u^2)=0,\quad{\textit{i.e.}}\quad\partial_t u=[u, |\nabla|] u.
\end{equation}
In \cite{Im1} global classical solutions starting from smooth positive data were constructed, and global weak solutions starting
from positive data in $L^\infty$. 
In \cite{Im2}, the authors established Schauder estimates for a general integro-differential equations,
which can be applied to \eqref{NB}.
The equation~\eqref{NB} bears a strong resemblance to classical inviscid models of hydrodynamics. For example, the
standard (local) Burgers equation can also be written in the form of a commutator:
\begin{equation}\label{burgers}
\partial_t u+\frac{1}{2}\partial_x(u^2)=0 \qquad\textit{i.e.}\qquad \partial_t u=[u, \partial_x]u.
\end{equation}
Thus the model \eqref{NB} can be seen as a variant of~\eqref{burgers} where $\partial_x$ 
is replaced by the non-local operator $|\nabla|$ of the same order. 
Similarly, if one considers the classical incompressible Euler equation
\begin{equation}\label{euler}
\partial_t \mathbf{u} + (\mathbf{u}\cdot\nabla) \mathbf{u}+\nabla P=0,
\end{equation}
where $P$ is the associated pressure given by $P=\Pi(\mathbf{u}\otimes \mathbf{u})$
and $\Pi=-(-\Delta)^{-1}\div^2$ is a singular integral operator with an even symbol,
we  can draw an analogy between terms: $(\mathbf{u} \cdot\nabla) \mathbf{u} \sim -u|\nabla|u$ and  $\nabla P \sim |\nabla|(u^2).$
Because of those formal analogies with Euler and Burgers, the model \eqref{NB} was  named  the \textbf{non-local Burgers equation}
in~\cite{Im1}, or (NB) for short. There, it was shown that the energy density $w=u^2$ plays a special role in the theory of well-posedness (see also~\eqref{GNB-w} below), which comforts the hydrodynamical flavor of this toy model.

One of the most interesting feature of~\eqref{NB} is its dual nature regarding the energy balance.
On the one hand, the $L^2$ energy $\|u\|_{L^2}^2$ is conserved (at least formally) because $\left( [u, |\nabla|] u \vert u \right)_{L^2} = 0$.
On the other hand, the fluctuations $v=u-\frac{1}{|\T^d|}\int_{\T^d}u$ satisfy
\begin{equation}\label{NBfluc}
\partial_t v + \frac{1}{|\T^d|}\left(\int_{\T^d}u(t,x) dx\right) |\nabla| v = [v,|\nabla|]v - \frac{1}{|\T^d|}\int_{\T^d} \left||\nabla|^{1/2} v\right|^2,
\end{equation}
which is a non-linear heat equation (of order 1) whose diffusion coefficient
is given by the average momentum of $u$. In turn, this average is controlled by
\[
\int_{\T^d}u_0(x) dx  + \int_0^t  \int_{\T^d} \left||\nabla|^{1/2} v(\tau,x)\right|^2 dx d\tau =
\int_{\T^d}u(t,x) dx \leq |\T^d|^{1/2} \|u_0\|_{L^2}.
\]
Formally, whenever $v$ starts to develop high-frequency structures, the value of the diffusion coefficient will increase and
possibly contain those structures.
This nonlinear feedback loop hidden in the conservative form of~\eqref{NB} suggests that the instabilities in the negative regions
may only be transient and will resorb themselves before developing a full blown singularity, at least if the average momentum
is positive.
On the contrary, examples of blow-up in finite time have been provided in~\cite{Im1}, where solutions
stemming from smooth negative data end up discontinuous at a later time, even though they remain bounded both
in~$L^1\cap L^\infty(\T^d)$
and in $L^2_t \dot{H}^{1/2}_x$.
The unsigned regime of the Non-local Burgers model(s) will not be further addressed in this article; its connections with
hydrodynamic turbulence will be the explored in later works.

\medskip
Let us conclude this brief tour of the litterature by mentioning
another class of non-local variants of the Burgers equation that was considered in~\cite{BCT2010},
and which includes for example
\begin{equation}\label{TNB1}
\partial_t v - 2\pi \partial_x (Hv)^2 = 0 \qquad x\in\R
\end{equation}
where $Hv=\frac{1}{\pi}\int_\R \frac{v(y)}{x-y} dy$ is the Hilbert transform (\textit{i.e.} the Fourier multiplier of symbol $-i\operatorname{sign}\xi$). More generally, \cite{H89} and \cite{BCT2010} consider models of the form
\begin{equation}\label{TNB2}
\partial_t v + \partial_x \mathcal{Q}(v) = 0 \qquad\text{with}\qquad
\mathcal{Q}(v) = \frac{1}{2\pi}\iint_{\R^2} e^{i x\xi} \Lambda (\xi-\eta, \eta) \hat{v}(\xi-\eta) \hat{v}(\eta) d\eta d\xi
\end{equation}
and where the symbol $\Lambda$ is symmetric, homogeneous of degree 0,
has the appropriate symmetries to map real valued functions into real valued ones
and is as smooth as the homogeneity allows. 
The models~\eqref{TNB2} appear, for example, in the description of surface acoustic waves in elasticity \cite{H89}.
Discarding technical assumptions, the general idea is that the well-posedness of~\eqref{TNB2} in high-regularity Sobolev
spaces is tied to the property:
\begin{equation}\label{TBNcrit}
\Lambda(1,0^+) = \Lambda(-1,0^+).
\end{equation}
On the contrary,~\eqref{TNB1} and, more generally, models of the form~\eqref{TNB2} that
do not satisfy this criterion will fail to have $C([-T,T],H^4(\R))$ solutions
for a dense subset of initial data in $H^4(\R)$. In this setting, the corresponding symbol for (NB) is
\[
\Lambda(k,\ell) = \frac{-i(k+\ell)}{|k+\ell|^2}\left( \frac{|k|+|\ell|}{2} - |k+\ell|\right).
\]
This symbol obeys~\eqref{TBNcrit} because $\Lambda(\pm 1,0)=0$ but it does not satisfies the regularity assumption (iv) in~\cite{BCT2010} 
because it is not bounded along the diagonal $k+\ell=0$.
Moreover, \eqref{TNB2} conserves the average of $v$ but it does not conserve, in general, the total $L^2$ energy.
To fit~\eqref{NB}-\eqref{NBfluc} into this class of models would require the addition of a non-linear damping term
to~\eqref{TNB2} to restore the global energy balance.

\medskip
In this paper we study the following 
\textbf{generalized Non-local Burgers equation} (GNB):
\begin{align}
&\partial_t u=[F(u),  |\nabla|^s] u,\quad x\in\mathbb{R}^d ~{\rm{or}} ~ \mathbb{T}^d,\label{GNB}\\
&u|_{t=0}=u_0 > 0\label{GNB-initialdata},
\end{align}
where $s\in(0, 1]$ and $|\nabla|^s=(-\Delta)^{s/{2}}$ denotes the fractional Laplacian. The  function $F:\R\to\R$ is given; one
will assume it to be $\mathcal{C}^{\infty}(\mathbb{R}^+)$  with $F'>0$  {\rm{a.e.}} on $[0,\infty)$.
For example, all the functions $F(u)=u^{n}$ ($n\in\mathbb{N}^\ast$)
or $F(u)=u-\sin u,$ $F(u)=e^{u}$,\ldots{} are admissible choices.
One can easily check that the proofs given in this article still work for $F(u)=u^\alpha$
for a real exponent~$\alpha>0$ (on $\T^d$) or $\alpha\geq1$ (on $\R^d$),
even though those functions fail to be $\mathcal{C}^{\infty}(\mathbb{R}^+)$ because they only have finitely many (if any)
bounded derivatives at the origin.
In the sequel, we will additionnaly assume that $F(0)=0$ because \eqref{GNB} is invariant when $F(u)$ is replaced by $F(u)-F(0).$
Notice that if $F'\equiv0$ then~\eqref{GNB}-\eqref{GNB-initialdata} would boil down to a trivial evolution $u(t)=u_0$.

\medskip
When $F(u)=|u|^{n-1}$  on $\R^+$, the (GNB) model \eqref{GNB} can be seen
 as a  type of porous medium equation of fractional order:
\begin{equation}\label{PME}
\partial_t u+|\nabla|^s (|u|^{n-1}u)=f
\end{equation}
 with a special source term $f=|u|^{n-1}|\nabla|^su$ that could model some forms of reaction or absorption of
 the density~$u$. The existence, uniqueness and  regularity problems of the homogeneous version of~\eqref{PME} 
have been fully investigated in~\cite{JL1, JL2}. The homogeneous problem for a general smooth increasing
non-linearity $F$ is adressed in~\cite{JL3} and allows for unsigned solutions.
We refer to~\cite{V007} for an in-depth coverage on the (local) porous medium equations
$\partial_t u - \Delta(u^m) = 0$ and to  \cite{BFR} for the fractional equivalent.
Let us point out that in some other models of porous media, the fractional derivative of~\eqref{PME} can also be modified
into a more geometric form $\div(|u|^{n-1}\nabla (|\nabla|^{s-2} u))$, as is the case in~\cite{Im4}; see also~\cite{V06}.
The connection of~\eqref{GNB} with these porous media models justifies our interest for positive solutions.

\medskip\noindent
At a formal level, the (GNB) model \eqref{GNB} admits the following structure properties.
\begin{itemize}
\item Translation invariance: if $t_0>0,~ x_0\in\mathbb{R}^d$ then $u(t+t_0, x+x_0)$ is another solution. In particular, the periodicity of the initial condition is preserved. 
\item Time reversibility: if $t_0 > 0$, then $-u(t_0- t, x)$ is a solution associated to $\tilde{F}(u) =-F(-u)$.
In particular, when $F$ is odd, it is a solution of the same equation.
\item Max / Min principle: if $u > 0$, then its maximum is decreasing and its
minimum is increasing. This follows most naturally from the representation~\eqref{GNB-integro-differential} below.
In particular, if $u_0>0$ then $u$ remains  positive at later times.
\item Energy conservation: $\|u(t)\|_{L^2} = \|u_0\|_{L^2}$ is obtained, for smooth $u$, by testing \eqref{GNB} against $u$ and
using the self-adjointness of $|\nabla|$.
\item Scaling invariance: if $u$ is a solution on $\R^d$, then $u_\lambda(t,x)=u(\lambda^s t,\lambda x)$ is a solution too
for any $\lambda>0$. On $\T^d$, the scaling transform makes sense only if $\lambda\in \N^\ast$ (quantified concentrations).
\end{itemize}
The updated version of the equation~\eqref{NBfluc} on the fluctuations $v(t,x)=u(t,x)-p(t)$ is more involved:
\begin{equation}\label{GNBfluc}
\partial_t v + p |\nabla|^s G_p v = [G_p v,|\nabla|^s]v - \frac{1}{|\T^d|}\int_{\T^d} G_pv \cdot |\nabla|^s v
\end{equation}
where $ p(t) = \frac{1}{|\T^d|}\int_{\T^d}u(t,x) dx$, $p'(t)=\frac{1}{|\T^d|}\int_{\T^d} G_pv \cdot |\nabla|^s v$ and $G_p(v) = F(v+p)-F(p)$.
The non-linear regularizing effect of~(GNB) is less striking on this formulation.
However, in the strongly positive regime, \textit{i.e.} if one assumes that $|v|\ll p$, one has $G_p(v)\simeq F'(p) v$
in which case~\eqref{GNBfluc} boils down to an equation with a structure similar to that of~\eqref{NBfluc}
and for which a similar heuristic can be expected.

\bigskip
Let us recall that, in $\mathbb{R}^d$ and for $s\in(0,1]$, the operator $|\nabla|^s$ can be defined as a singular integral :  
\begin{equation}\label{kernel_rep}
|\nabla|^s f :={\rm{p.v. }}\int_{\mathbb{R}^d}\bigl(f(x)-f(y)\bigr)K^s(x-y)\,dy
\end{equation}
 with a kernel $K^s(z)=c_{d, s}|z|^{-d-s}$ where $c_{d,s}>0$ is a constant depending on the dimension $d$ and on $s$.
 When $s=1$, the numerator of the singular integral is corrected into $f(x)-f(y)-(x-y)\cdot\nabla f(x)$ to restore
 integrability near the diagonal, \textit{i.e.}~one considers Hadamard's finite part instead of Cauchy's principal value.
 We refer to~\cite{Kwa} for other equivalent definitions of the fractional Laplace operator.
Thanks to its commutator structure, the (GNB) model \eqref{GNB} can be rewritten in the following integral form:
 \begin{equation}\label{GNB-integro-differential}
 \partial_t u={\rm{p.v. }}\int_{\mathbb{R}^d}\Bigl(F(u(t, y))-F(u(t, x))\Bigr)u(t, y)K^s(x-y)\,dy.
 \end{equation}
If $u$ is periodic with a period $2\pi$ in all the coordinates, the representation~\eqref{GNB-integro-differential}
becomes
\begin{equation}\label{GNB-integro-differential-periodic}
 \partial_t u={\rm{p.v. }}\int_{\mathbb{T}^d}\Bigl(F(u(t, y))-F(u(t, x)\Bigr)u(t, y)K_{{\rm{per}}}^s(x-y)\,dy,
 \end{equation}
where $\mathbb{T}^d$ is the torus and $K_{{\rm{per}}}^s(z)=\sum_{j\in\mathbb{Z}^d}\dfrac{c_{d, s}}{|z+2\pi j|^{d+s}}.$

In the periodic framework, both representations are valid
due to a sufficient decay of $K^s$ at infinity, while the former~\eqref{GNB-integro-differential} can, from time to time, be more amenable to
an analytical study due to the explicit nature of the kernel and the applicability of known results.

Let us point out that all our results are proved in the \textbf{periodic setting}, except for the local existence, which holds
in both the periodic and the open case. The periodicity provides extra compactness of the underlying domain,
which, for positive data and in conjunction with the minimum principle, warrants uniform bounds away from zero
in space and time that further entail uniform ellipticity of the right-hand side of \eqref{GNB-integro-differential}.  
In the whole space, the finiteness of the energy prevents uniform lower-bounds; suitable lower barriers
are not readily available either.

\bigskip
\paragraph{\bf Main results.}
Inspired by \cite{Im1, Im2}, the aim of this paper is to develop a well-posedness theory for the Generalized
Non-local Burges model \eqref{GNB} and to study its long-time behaviour.
We make use of the dual nature of (GNB) explained above, \textit{i.e.} both globally conservative and
with dissipative fluctuations, by way of blending classical techniques relevant to the Euler equation \cite{Majda},
such as energy estimates and a Beale-Kato-Majda (BKM) criterion, with recently developed tools of the regularity theory
for parabolic integro-differential equations \cite{Caf2011, Con2012, Fe2013, Jin1,  Miku}.

\medskip\noindent
Let us now give a brief summary of our results. While our results are, overall, quite similar in nature to those of~\cite{Im1, Im2},
we would like to point out that the use of a weaker non-local derivative $|\nabla|^s$ combined with a wilder non-linearity~$F(u)$
required careful technical adaptations at many critical moments
and that the persistence of most statements came to us somewhat as a surprise.
See in particular the novel estimates~\eqref{BKM-use}, \eqref{estimate-1111} and \eqref{BKM-en}  involving $F'(u)$,
the fact that Theorem~\ref{Th-Schauder-estimate} and Appendix~\ref{SCH} span fractional regularities $0<s\leq1$,
the use of an a-priori bound~\eqref{conservation-P} in~$L^p$ and Theorem~\ref{thm-stab} on stability.
\begin{itemize}[leftmargin=*]
\item  \textbf{Local existence with a BKM criterion.} 
Let $s\in (0,1]$.
For initial data $u_0 \in H^m(\Omega^d)$ on $\Omega^d = \mathbb{R}^d~ {\rm{or}}~ \mathbb{T}^d$, with $u_0 > 0$ pointwise and $m>\frac{d}{2}+1$, there exists a unique local solution of~\eqref{GNB} in
\[
\mathcal{C}([0, T);~H^m(\Omega^d))\cap \mathcal{C}^1([0, T); H^{m-1}(\Omega^d)).
\]
Even for this local existence result, the positivity of the initial data is essential. We also have a Beale-Kato-Majda regularity criterion: if $\int_0^T \|\nabla u(t)\|_{L^\infty}\,dt < \infty$, the solution extends smoothly beyond $T$. The proof goes via a smoothing scheme based on a regularization of the kernel (\S\ref{Local well-posedness}, \ref{A Beale-Kato-Majda criterion}; Theorems~\ref{Th_LWP} and~\ref{Th-BKM}).
\item\textbf{Instant regularization and global well-posedness.}
 Any positive classical solution to \eqref{GNB} on a time interval $[0, T )$ satisfies uniform bounds:
 for any $k \in\mathbb{N}$,   for any $0<t_0 <T$ and any $0<s_0\leq s\leq 1$:
\begin{equation}\label{higher-Holder-estimates-1}
\| \partial_t^k u, \nabla^k_x u\|_{L^\infty_{t, x} ((t_0, T)\times\mathbb{T}^d)} \leq C(d, s_0,  k, t_0, T, \min u_0, \max u_0).
\end{equation}
To achieve this (\S\ref{Instant regularization}; Theorem~\ref{Th_GWP})
we symmetrize the right-hand side of \eqref{GNB-integro-differential} by multiplying it by $2u$ and use
\[
F(u(t, y))-F(u(t, x))=(u(t, y)-u(t, x))\int^1_0 F'\bigl((1-\lambda)u(x)+\lambda u(y)\bigr)\,d\lambda,
\]
to write the evolution equation for the energy density $w = u^2$:
 \begin{equation}\label{GNB-w}
 \partial_t w={\rm{p.v. }}\int_{\mathbb{R}^d}\bigl(w(y)-w(x)\bigr) \mathcal{K}^s(t, x, y)\,dy
 \end{equation}
with
\begin{equation}\label{kernal-w}
\qquad\qquad\mathcal{K}^s(t, x, y)=\frac{c_{d, s}}{|x-y|^{d+s}}\frac{2 u(x)u(y)}{u(x)+u(y)}\int_0^1 F'\bigl((1-\lambda)u(x)+\lambda u(y)\bigr)\,d\lambda.
\end{equation}
The active kernel $\mathcal{K}^s$ is symmetric and satisfies uniform ellipticity bound
$\frac{\Lambda^{-1}}{|x-y|^{d+s}} \leq \mathcal{K}^s \leq \frac{\Lambda}{|x-y|^{d+s}}$.
This puts the equation~\eqref{GNB-w} within the range of recent results
of Kassmann et al. \cite{Bar2009, Ka2009} and of Caffarelli-Chan-Vasseur \cite{Caf2011}
where De Giorgi-Nash-Moser techniques were adopted; this yields an initial H\"{o}lder regularity for $w$
and hence for $u$ by positivity (and some functional analysis).
To obtain the  bounds \eqref{higher-Holder-estimates-1}, we then follow the idea of \cite{Im2} to get a
Schauder estimate for a class of parabolic integro-differential equations with a general fractional kernel (Theorem~\ref{Th-Schauder-estimate});
see also \cite{Jin1, Miku}.
At this point, it readily follows from the Beale-Kato-Majda criterion and the instant regularization
property that smooth solutions exists globally in time. 
\item\textbf{Global existence of weak solutions.}   
Since the bounds \eqref{higher-Holder-estimates-1} depend
essentially only on the $L^\infty$ norm of the initial condition, we can construct
a sequence of global smooth approximate solutions by smoothing out any initial data $u_0 \in L^\infty(\mathbb{T}^d), u_0>0$.
These solutions  enjoy an a-priori bound in the space
$ L^\infty(\mathbb{R}^+ \times \mathbb{T}^d)\cap L^2(\mathbb{R}^+; \dot H^{{s}/{2}}(\mathbb{T}^d))$
and one can prove compactness, extract a subsequence and  prove
that the weak limit still satisfies the~(GNB) equation (\S\ref{weak-solu}; Theorem~\ref{Th_weaksolu}).
As a corollary, we show by time-reversal duality that, if $F$ is odd,
some negative smooth initial data can develop a first singularity in finite time.
\item\textbf{Long-time asymptotics.}  Any weak solution to \eqref{GNB} 
converges to a constant, namely $|\mathbb{T}^d|^{-1/2}\|u_0\|_{L^2(\mathbb{T}^d)},$
 in the following strong sense: the oscillation (\textit{i.e.} amplitude) of $u(t)$ and the
 semi-norm~$\|\nabla u(t)\|_{L^\infty}$ tend
 to 0 exponentially fast with some delay for the convergence of small-scale features (\S\ref{sec3};
 Theorems~\ref{decay-A} and~\ref{Th-decay}).
 A stability result with respect to the nonlinearity $F$ is also presented (Theorem~\ref{thm-stab}).
\end{itemize}

\medskip
The article is organized as follows. All results pertaining to the well-posedness of~\eqref{GNB} are presented in~\S\ref{sec2}
and the gradual steps are organized in subsections. In turn,~\S\ref{sec3} is devoted to the long-time asymptotics of solutions.
Appendix \ref{LPB} contains a brief primer on the Littlewood-Paley theory and ensures that this
article is mostly self-contained.
Appendix \ref{SCH} details the proof of the Schauder estimates that generalize~\cite{Im2}
and that could be of interest on their own for other applications.

\bigskip
\paragraph{\bf Notations.}
We end this introductory part with a few notations that are used throughout the article.
We denote by $C$ a harmless positive constant that  may change from one line to the next,
and we write $A\lesssim B$ instead of $A\leq C B$.
The Euclidean ball in $\mathbb{R}^d$ with center $x$ and radius $r$ is denoted by $B_r (x )$.
For $X$ a Banach space, $p\in[1, \infty]$ and $T\in(0,\infty]$, the notation $L^p(0, T; X)$  designates the set of measurable functions $f: [0, T]\to X$ with $t\mapsto\|f(t)\|_X$ in $L^p(0, T)$, endowed with the norm $\|\cdot\|_{L^p_{T}(X)} :=\|\|\cdot\|_X\|_{L^p(0, T)}.$ For any interval $I$ of $\R,$ we agree that $\mathcal C(I; X)$ denotes the set of continuous functions from $I$ to $X$.
For any $\alpha, \beta\in(0, 1]$, we define the H\"older semi-norm as follows:
\begin{equation}\label{holder_seminorm}
[f]_{\mathcal{C}^{\alpha, \beta}_{t, x}(I\times\mathbb{R}^d)}:=
\sup\left\{\frac{|f(t, x)-f(\tau, y)|}{|t-\tau|^\alpha+|x-y|^\beta} \,;\, (t, x), (\tau, y)\in I\times\mathbb{R}^d, (t, x)\neq(\tau, y)\right\}.
\end{equation}
We denote by $\mathcal{C}^{\alpha, \beta}_{t, x}(I\times\mathbb{R}^d)$ the H\"older space, which is equipped with the norm
$$\|f\|_{\mathcal{C}^{\alpha, \beta}_{t, x}(I\times\mathbb{R}^d)}:= \|f\|_{L^\infty(I\times\mathbb{R}^d)}+[f]_{\mathcal{C}^{\alpha, \beta}_{t, x}(I\times\mathbb{R}^d)}.$$
For  any nonnegative integers $n_1$ and $n_2$, the norm
\begin{equation}\label{holder_seminorm2}
\|f\|_{\mathcal{C}^{n_1+\alpha, n_2+\beta}_{t, x}(I\times\mathbb{R}^d)}:= \|f\|_{L^\infty(I\times\mathbb{R}^d)}+[\partial_t^{n_2}f]_{\mathcal{C}^{\alpha, \beta}_{t, x}(I\times\mathbb{R}^d)}+[\nabla_x^{n_2}f]_{\mathcal{C}^{\alpha, \beta}_{t, x}(I\times\mathbb{R}^d)}
\end{equation}
define  the space $\mathcal{C}^{n_1+\alpha, n_2+\beta}_{t, x}(I\times\mathbb{R}^d).$ 
Sometimes, we omitted the subscript $t, x,$ respectively.

\bigskip
\paragraph{\bf Thanks.}
The authors would like to thank 
the Université Paris-Est Creteil (LAMA, UMR 8050 CNRS) where we both used to work until recently.


\section{\bf Global well-posedness with positive periodic initial data}\label{sec2}
\subsection{Local well-posedness with positive initial data}\label{Local well-posedness}

We start with our discussion with the local well-posednesss in high-regularity classes.
In this section, $\Omega^d$ denotes either $\mathbb{R}^d$ or $\mathbb{T}^d.$
\begin{theorem}\label{Th_LWP}
Let $m>\frac{d}{2}+1$ be an integer. Given a pointwise \textbf{positive} initial data $u_0\in H^m(\Omega^d),$ then there exists a time $T>0$ such that there exists a unique local solution
\[
u\in \mathcal{C}([0, T); H^m(\Omega^d))\cap \mathcal{C}^1([0, T); H^{m-1}(\Omega^d))
\]
to the (GNB) Cauchy problem \eqref{GNB}-\eqref{GNB-initialdata}. Moreover, $u(t, x)>0$ for all $(t, x)\in([0, T)\times\Omega^d,$
and the maximum~$\max\limits_{x\in\Omega^d}\,u(t, x)$ is strictly decreasing in time.
\end{theorem}
\begin{remark}\sl
In the case of $\Omega^d=\mathbb{T}^d$, we have a complementary statement for the minimum: $\min\limits_{x\in\mathbb{T}^d}\, u(t, x)$ is a strictly increasing  function of time, thus the maximum oscillation of $u$ is shrinking.
In section \S\ref{sec3}, we will prove more precise statements on the asymptotic behaviour of the amplitude.
\end{remark}

\smallskip
The proof in the case of $\Omega^d=\mathbb{R}^d$ requires slightly more technical care about the usage of the maximum principle,
while being similar in the rest of the argument. We therefore present it only in the case of $\mathbb{R}^d$.

\begin{proof}
The proof is based on a classical energy method, which requires a regularization of the kernel. 
We will split it  into five successive steps.
\medbreak\noindent
\textit{Step 1: Regularization.} Given $\delta\in(0, 1].$ Let us consider the following regularization of the kernel on the
spectral side
\begin{equation}\label{regu-kernel}
\widehat{K^s_{\delta}}(\xi):= \int_{\R^d} e^{-i\xi\cdot y}K^s_\delta(y)\,dy=\frac{1}{\delta}e^{-\delta|\xi|^s}
\end{equation}
and the corresponding operator
\begin{equation}\label{def-frac-operator}
|\nabla|^s_{\delta} f:=\int_{\mathbb{R}^d}\bigl(f(x)-f(y)\bigr)K^{s}_{\delta}(x-y)\,dy= \widehat{K^s_{\delta}}(0)f-T^{s}_{\delta} f=\frac{1}{\delta}f-T^{s}_{\delta} f
\end{equation}
where $T^{s}_{\delta} f=K^{s}_{\delta}\star f$ is a convolution. 
Note that $T^{s}_{ \delta}$ is infinitely smoothing since its symbol is exponentially  decreasing;
in particular $\|T^{s}_{\delta} u\|_{H^{m}}\leq C_{d, s, \delta}\|u\|_{H^{m}}.$ 

\begin{remark}\sl
From \cite{BG60}, one collects an explicit formula for the regularized kernel: 
\begin{equation*}
K^{s}_{1}(y)=\frac{1}{(2\pi)^{\frac{d}{2}}|y|^{\frac{d}{2}-1}}\int^\infty_0e^{-t^s}t^{\frac{d}{2}}J_{\frac{d-2}{2}}(|y|t)\,dt
\end{equation*}
where $
J_{\alpha}$ denotes the Bessel function of first kind of order $\alpha$.
In particular, $K^{s}_{1}(y)$ is a continuous strictly positive radial function on $\mathbb{R}^d.$
Using the scaling invariance of the Fourier transform, one gets:
$$K^s_{\delta}(y)=\delta^{-(\frac{d}{s}+1)}K^{s}_{1} \left( \delta^{-\frac{1}{s}} y \right).$$
When $s=1$, one recovers the formula
$$K^1_\delta(y)=\frac{c_{d, 1}}{(\delta^2+|y|^2)^\frac{d+1}{2}}$$
from \cite{Im1} because the Fourier transform
$$\int_{\R^d} e^{-iy\cdot\xi -\delta|\xi|}\,d\xi=(2\pi)^d \frac{c_{d, 1}}{(\delta^2+|y|^2)^\frac{d+1}{2}}\cdotp$$
exchanges the Abel and the Poisson kernels.
\end{remark}

\medskip
The regularized version of the equation~\eqref{GNB} takes the form
\begin{align}
\partial_t u&=[F(u), |\nabla|^s_{\delta}]u\label{regu-eq1}\\
&=\int_{\mathbb{R}^d}\bigl(F(u(y))-F(u(x))\bigr)u(y)K^{s}_{\delta}(x-y)\,dy\label{regu-eq2}\\
&=-[G(u), T^{s}_{\delta}]u\label{regu-eq3}
\end{align}
where $G(u):=F(u)-F(0).$
In what follows, each of the three forms of this equation will play a role.
Let us point out that the commutator structure of (GNB) eliminates the unbounded term in~\eqref{def-frac-operator}.

\bigskip
Before going further, we would like to recall a kind of composition lemma based on Meyer’s first linearization method,
that has been wildly used in compressible fluid dynamics when the pressure law depends on the density of the fluid
(see \textit{e.g.}~\cite{Da00} for an application to the well-posedness of compressible Navier-Stokes equations
in the setting of critical Besov spaces).
We state a version of the lemma that holds in Sobolev spaces;
the proof and various generalization can be found in e.g.  \cite{DaNOTE, Ba11, TW}. 
\begin{lemma}\label{compo-lemma} (Proposition 1.4.8 in \cite{DaNOTE})
Let $I$ be an open interval of\, $\mathbb{R}$ and $J$ a compact subset.
Let $r>0$ and $\sigma$ be the smallest integer such that $\sigma\geq r$.
If $\mathcal{G} : I\to\mathbb{R}$ satisfies $\mathcal{G}(0)=0$ and $\mathcal{G}'\in W^{\sigma, \infty}(I; \mathbb{R})$
and $f\in H^r\cap L^\infty$ has values in $J$, then $\mathcal{G}(f)\in H^r$ and there exists a constant $C_1$
depending only on $r, I, J, d$  
such that
\[
\|\mathcal{G}(f)\|_{H^r}\leq C_1 
(1+\|f\|_{L^\infty})^\sigma
\|\mathcal{G}'\|_{W^{\sigma, \infty}(I)}\|f\|_{H^r}.
\]\end{lemma}
\begin{lemma}\label{compo-lemma2} (Corollary 1.4.9 in \cite{DaNOTE})
Let $I$ be an open interval of \,$\mathbb{R}$ and $J$ a compact subset.
Let $r>d/2$ and $\sigma$ be the smallest integer such that $\sigma\geq r$.
If $\mathcal{G} : I\to\mathbb{R}$ satisfies $\mathcal{G}(0)=0$ and $\mathcal{G}''\in W^{\sigma, \infty}(I; \mathbb{R})$
and $f, g\in H^r\cap L^\infty$ have values in $J$,
then  there exists a constant $C_2$ depending only on $r, I, J, d$ 
such that
\begin{align*}
\|\mathcal{G}(f)-\mathcal{G}(g)\|_{H^r}
\leq C_2
(1+\|f\|_{L^\infty})^\sigma
 \|\mathcal{G}''\|_{W^{\sigma, \infty}(I)}
&\Bigl(\|f-g\|_{H^r}\sup_{\tau\in[0, 1]}\|(1-\tau)f+\tau g\|_{L^\infty}\\
&\quad+\|f-g\|_{L^\infty}\sup_{\tau\in[0, 1]}\|(1-\tau)f+\tau g\|_{H^r}\Bigr).
\end{align*}
\end{lemma}

Using Lemma \ref{compo-lemma},  we will now show that  the right-hand side of \eqref{regu-eq3}
is quadratically bounded and locally Lipschitz on any open set of $H^m$. To that effect, let us introduce
the open ball
$$B_M:=\{u\in H^m \,;\, \|u\|_{H^m}< M\}$$
and recall that $H^{m}(\mathbb{R}^d)$ is an algebra with $H^{m}(\mathbb{R}^d)\hookrightarrow L^\infty(\mathbb{R}^d)$.
For any $u, v\in B_M$, a standard quadratic estimate (\textit{i.e.} discarding the commutator structure) reads:
\begin{align*}
\|[G(u), T^{s}_{\delta}]u\|_{H^m}
\leq 2 C_{s,d,\delta} \|u\|_{H^m}  \|G(u)\|_{H^m}
\leq C_\delta \|u\|_{H^m}^2
\end{align*}
where $C_\delta$ is a constant that depends on $s$, $d$, $m$, $\delta$, $\|u\|_{L^\infty}$ and $\|F'\|_{W^{m,\infty}}$.
Similarly, the constant $C_\delta$ may be adjusted to incorporate $\|F''\|_{W^{m,\infty}}$ and to ensure that: 
\begin{align*}
\|[G(u), T^{s}_{\delta}]u-[G(v), T^{s}_{\delta}]v\|_{H^m}
&\lesssim \|G(u) T^{s}_{\delta} u - G(v) T^{s}_{\delta} v\|_{H^m}  + \|u G(u)-vG(v)\|_{H^m} \\
&\lesssim \|u-v\|_{H^m} \|G(u)\|_{H^m}  + \|v\|_{H^m} \|G(u)-G(v)\|_{H^m} \\
&\leq C_\delta \|u-v\|_{H^m} (1+\|u\|_{H^m}+\|v\|_{H^m})^2.
\end{align*}
Picard's theorem on Banach spaces (see \textit{e.g.}~\cite{C1995}, \cite{Majda}) implies that for any $u(0, x)\in B_M$,
there is a unique local solution $u\in\mathcal{C}^1([0, T); B_M)$ to \eqref{regu-eq3};
here $T$ may depend on $\|u\|_{H^m}$ and $\delta.$
For later use, note that the energy $\|u(t)\|_{L^2}=\|u(0, x)\|_{L^2}$ is conserved,
because $u$ is a legitimate multiplier for~\eqref{regu-eq1}.
\begin{remark}\sl
This step requires $m+2$ derivatives of $F$ to be bounded only on the set of values taken by $u$.
In the periodic case, the max/min principle bellow thus allows for $F(u)=u^\alpha$ for any $\alpha>0$ while
on $\R^d$, caution should be taken in $\alpha\not\in\N^\ast$; in that case, one would have
to regularize $F$ into $F_\vartheta = (\vartheta+u^2)^{\alpha/2}$ and pass to the limit $\vartheta\to0$,
provided uniform bound with respect to $\vartheta$ in the subsequent steps.
\end{remark}

\medbreak\noindent
\textit{Step 2: Maximum principle.} Suppose, in addition, that $u(0, x)>0.$
Let $u$ be the corresponding local solution to \eqref{regu-eq3} in $\mathcal{C}^1([0, T); H^m(\mathbb{R}^d))$. 
As $m>\frac{d}{2}+1,$ the function $u(t,\cdot)\in H^m(\mathbb{R}^d)$ is continuous and tends to zero at infinity
thus it attains its maximum $M(t)=\max_{x\in\mathbb{R}^d} u(t, x)$.
We claim in this section that $u(t, x)>0,$ for all $(t, x)\in[0, T)\times \mathbb{R}^d$ and that the maximum function $M(t)$
is strictly decreasing on $[0, T).$ 

\medskip
Let us prove the positivity first.
Let us fix $R>0$ and show that $u$ never vanishes on $(0, T)\times B_R(0).$ Suppose it does. Let us consider
$$t_0:=\inf\{t\in(0, T):\,\exists\, x\in B_R(0),\,s.t.\enspace u(t, x)=0\}.$$
The compactness of $[0, T]\times \overline{B_R(0)}$ and the continuity of $u$ ensure that $t_0$ is attained. 
Since $u_0>0,$ then $t_0>0.$  We next show that $ u(t_0, x)\geq 0$ for all $x\in B_R(0)$; if it was not the case,
then an $x_\star\in B_R(0)$ would exist such that $u(t_0, x_{\star})<0$.
 Thanks to the continuity of $u$, there exists a constant $\eta>0$ such that
 \[
 \forall (t, x)\in(t_0-\eta, t_0)\times B_\delta(x_\star), \qquad
 |u(t, x)-u(t_0, x_\star)|\leq \frac{1}{2}|u(t_0, x_\star)|
 \]
and in particular
$$u\left(t_0-\frac{\eta}{2},  x_\star\right)\leq \frac{1}{2}u(t_0, x_\star)<0.$$
This is a contradiction to the definition of $t_0.$ Thus $ u(t_0, x)\geq 0$ for all $x\in B_R(0)$
and, by continuity, also for all $x\in \overline{B_R(0)}$.

In the case of $\T^d$, this argument (with $\T^d$ used instead of $\overline{B_R(0)}$) is sufficient
to ensure the positivity of the solution.
For the full space, let $x_0\in \overline{B_R(0)}$ be such that $u(t_0, x_0)=0.$  Evaluating \eqref{regu-eq2} at $(t_0, x_0)$ we obtain
$$\partial_t u(t_0, x_0)=
\int_{\R^d}G(u(y)) u(y) K^s_{\delta}(x_0-y) dy=
\int_{\mathbb{R}^d}\int_0^1 F' (\lambda u(y))u^{2}(y)K^s_{\delta}(x_0-y)\,d\lambda\,dy.$$
On the right-hand side, one has yet no control over the sign of $u(y)$ when $|y|>R$.
If $F'>0$  {\rm{a.e.}} on $\R$ (for example if $F$ is odd and strictly increasing on $\R^+$), then the right-hand side is strictly positive;
otherwise, the conservation of energy of solutions of \eqref{regu-eq1} would imply $u\equiv 0$ on $(0,t_0)$.
This shows that $u(t, x_0)$ vanishes for some earlier time $t<t_0$, which is a contraction.
Since the argument holds for arbitrary large values of~$R$, the positivity claim follows.

\begin{remark} \sl
Let us give some details if $F'>0$ holds only {\rm{a.e.}}~on $[0,\infty)$, which is the case for even-extensions of $F$.
As observed above, the positivity statement holds without further effort on $\T^d$.
For $\R^d$, one should temporarily modify $F$ to ensure that it is increasing on $[-\varepsilon_0,\infty)$ for some $\varepsilon_0>0$.
Then, the previous argument works provided that  one restricts the initial choice of $R$ to values large
enough to ensure that $|u(t,y)|<\varepsilon_0$ when $|y|>R$, which establishes the positivity of the solution.
The alteration of $F(u)$ for negative values of $u$ can then be dropped as irrelevant.
\end{remark}

\medskip
Let us prove the second claim now; suppose that $M(t)$ is not strictly decreasing on $[0, T).$
This implies that there exists a pair of times $0<t_1<t_2<T$ such that $M(t_1) \leq M(t_2)$.
If $M(t_1) < M(t_2)$, then by the continuity of $M(t)$ (which follows from the fact that $u$ is continuous),
$M(t)$ attains its maximum on the interval $ [t_1 ,t_2]$. Choose $t_0 \in[t_1, t_2]$ be the left utmost point
where the maximum of $M(t)$ is attained. Then $t_0 >t_1$, and $M(t_0)\geq M(t)$ for all $t_1 \leq t\leq t_0.$
If, on the contrary $M(t_1) = M(t_2)$ then either one can shrink the interval to fulfil the previous assumption
or $M(t)$ is constant throughout $[t_1, t_2]$.
In either case, there exists a $t_0$ in $(t_1,t_2]$ such that $M(t_0) \geq M(t)$ on $[t_1, t_0]$.

Let us now consider a point $x_0 \in\mathbb{R}^d$ such that $u(t_0,x_0) = M(t_0)$. Then, provided $u\not\equiv0$:
$$\partial_t u(t_0, x_0)=\int_{\mathbb{R}^d}\bigl(F(u(y))-F(u(x_0))\bigr)u(y)K^{s}_{ \delta}(x_0-y)\,dy<0.$$
This implies that, at an earlier time $t<t_0,$ one must have $u(t, x_0)>u(t_0, x_0)=M(t_0)$ which in contradiction
with the initial assumption. One has thus established the strict decay of $M(t)$.

\medbreak\noindent
\textit{Step 3: uniform bounds.} Let us first state a uniform estimate in term of $\delta$:
\begin{equation}\label{estimate-frac}
\||\nabla|^s_\delta f\|_{H^r}\leq \||\nabla|^s f\|_{H^{r}}
\end{equation}
for all $r\in\mathbb{R}^+.$ 
Indeed, by virtue of \eqref{def-frac-operator} and the definition $\|f\|_{H^r}:=\|(1+|\xi|^2)^{\frac{r}{2}}\widehat{f}\|_{L^2}$
of nonhomogeneous Sobolev spaces (see \textit{e.g.}~\cite{Ba11}), one can write
\begin{align*}
\||\nabla|^s_\delta f\|_{H^r}&
=\left\|(1+|\xi|^2)^{\frac{r}{2}}\widehat{f}(\xi)\,\frac{(1-e^{-\delta|\xi|^s})}{\delta}\right\|_{L^2}\\
&\leq\|(1+|\xi|^2)^{\frac{r}{2}}|\xi|^s \widehat f (\xi)\|_{L^2}
\,\times\,\sup_{\delta,~\xi}\dfrac{|1-e^{-\delta|\xi|^s}|}{\delta|\xi|^s}
\leq\||\nabla|^s f\|_{H^r}.
\end{align*}
Even with other equivalent norms, the constant would remains uniform in $\delta$.

\medskip
We now use a technique inspired by the classical energy method employed to solve the Euler equation in high-regularity Sobolev spaces.
Let $\alpha$ be a multi-index of order $|\alpha| =m$.
Differentiating \eqref{regu-eq1}, we obtain
\[
\partial_t\partial^\alpha u= [F(u), |\nabla|^s_\delta]\partial^\alpha u+\sum_{0<\beta_1\leq\alpha}\partial^{\beta_1}(F(u))|\nabla|^s_\delta\partial^{\alpha-\beta_1} u
 -|\nabla|^s_\delta\Bigl(
\sum_{0<\beta_2<\alpha}\partial^{\beta_2} (F(u))\partial^{\alpha-\beta_2}u
+ u \partial^\alpha(F(u))
\Bigr).
\]
For multi-indexes $\alpha, \beta$, an inequality $0\leq\beta\leq\alpha$ means  $0\leq \beta_j\leq \alpha_j$ for each $j=1,\dots,d$.
A strict inequality~$\beta<\alpha$ means $\alpha-\beta\geq0$ with $\alpha\neq\beta$.
Next, one can expand the term $u \partial^\alpha(F(u))$
using the fact (since $m\geq2$) that there exists $j\in\{1\cdots d\}$  such that $\partial^\alpha=\partial^{\alpha-e_j}\partial_j$.
\begin{align*}
\partial_t\partial^\alpha u=&~[F(u), |\nabla|^s_\delta]\partial^\alpha u+\sum_{0<\beta_1\leq\alpha}\partial^{\beta_1}(F(u))|\nabla|^s_\delta\partial^{\alpha-\beta_1} u\\
& -|\nabla|^s_\delta\Bigl(\sum_{0<\beta_2<\alpha}\partial^{\beta_2} (F(u))\partial^{\alpha-\beta_2}u
\enspace+\!\!\!\! \sum_{0<\beta_3\leq\alpha-e_j}u\,\partial^{\beta_3}(F'(u))\partial^{\alpha-\beta_3} u
\enspace+\enspace u\, F'(u)\partial^{\alpha} u\Bigr).
\end{align*}
Let us take the $L^2$ inner product of the above equation with $\partial^\alpha u$, using
 the properties
\begin{equation}\label{property-IBP}
\int f|\nabla|^s_\delta g=\int g|\nabla|^s_\delta f
\qquad\text{and}\qquad
\int g\cdot [f, |\nabla|^s_\delta] g=0.
\end{equation}
The first term disappears and we have
\begin{align}
\frac{d}{dt}\|\partial^\alpha u\|_{L^2}^2\notag
=&~\int \partial^\alpha u\sum_{0<\beta_1\leq\alpha}\partial^{\beta_1}(F(u))|\nabla|^s_\delta\partial^{\alpha-\beta_1} u\notag- \sum_{0<\beta_2<\alpha}\int \partial^\alpha u \,|\nabla|^s_\delta\Bigl(\partial^{\beta_2} (F(u))\partial^{\alpha-\beta_2}u\Bigr)\notag\\
&\quad-\sum_{0<\beta_3\leq\alpha-e_j}\int \partial^\alpha u\,|\nabla|^s_\delta\bigl(u\,\partial^{\beta_3}(F'(u))\partial^{\alpha-\beta_3} u\bigr)-\int \partial^\alpha u\,|\nabla|^s_\delta\bigl(uF'(u)\partial^\alpha u \bigr)\label{energy-es}.
\end{align}
The last  term is the most singular one because it contains a derivative, which is formally of order $m+s$;
let us find an upper-bound for it first.
Using \eqref{property-IBP} and the definition of $|\nabla|^s_\delta,$  one has
\begin{align*}
-\int \partial^\alpha u\,|\nabla|^s_\delta\bigl(u F'(u)\partial^\alpha u\bigr)
=&-\int \bigl(|\nabla|^s_\delta \partial^\alpha u\bigr)\,\bigl(u F'(u)\partial^\alpha u\bigr)\\
=&-\int\int u(x)F'(u(x))\,\partial^\alpha u(x)\bigl(\partial^\alpha u(x)-\partial^\alpha u(y)\bigr)K^s_\delta(x-y)\,dx\,dy.
\end{align*}
Using the positivity of $u$ and $F'$, an upper bound for this term will now follow from the elementary identity 
\begin{equation}\label{id_rem}
-a(a-b) \leq -\frac{1}{2}\, (a^2-b^2 ).
\end{equation}
More precisely, as $m>\frac{d}{2}+1$, we use the embedding $H^{m-1}(\mathbb{R}^d)\hookrightarrow L^\infty(\mathbb{R}^d)$
and~\eqref{estimate-frac} to find that
\begin{align}
-\int \partial^\alpha u\,|\nabla|^s_\delta\bigl(u F'(u)\partial^\alpha u\bigr)
&\leq-\frac{1}{2}\int\int u(x) F'(u(x)) \bigl((\partial^\alpha u)^2(x)-(\partial^\alpha u)^2(y)\bigr)K^s_\delta(x-y)\,dx\,dy\notag\\
&=-\frac{1}{2}\int u(x) F'(u(x))  \,|\nabla|^s_\delta\bigl((\partial^\alpha u(x))^2\bigr) \,dx\notag\\
&=-\frac{1}{2}\int (\partial^\alpha u)^2 |\nabla|^s_\delta\bigl(u F'(u) \bigr)
\leq \frac{1}{2}\,\|\partial^\alpha u\|_{L^2}^2\,\||\nabla|^s_\delta\bigl(uF'(u) \bigr)\|_{L^\infty}\notag\\
& \lesssim \|u\|_{H^m}^2\,\| u F'(u)\|_{H^{m}}
\lesssim \|u\|_{H^m}^3.\label{estimate-111}
\end{align}
In the last step  we have applied the Lemma \ref{compo-lemma} to the smooth
function $u F'(u)$ and used the maximal principle proved in Step 2 to factor out $(1+\|u\|_{L^\infty})^m\leq (1+M(0))^m$
into the constant.

The rest of the expression \eqref{energy-es} is simpler to deal with as it does not contain any other derivatives of (formal) order $m +s.$ To estimate it, we will use Lemma \ref{compo-lemma}  and the Gagliardo-Nirenberg inequalities:
\begin{equation}\label{GN-ineq}
\|\partial^\gamma f\|_{L^{2r/|\gamma|}}\lesssim \|f\|_{L^\infty}^{1-\frac{|\gamma|}{r}}\|u\|_{H^{r}}^{\frac{|\gamma|}{r}},\qquad0\leq|\gamma|\leq r,
\end{equation}
and the following Kato-Ponce inequality (see \cite{Ke91}):
\begin{equation}\label{KP-ineq1}
\||\nabla|^s (fg)\|_{L^{2}}\lesssim \||\nabla|^s f\|_{L^p} \|g\|_{L^{\bar p}}+\|| f\|_{L^q} \||\nabla|^s g\|_{L^{\bar q}}
\end{equation}
for $p, \bar q\in [2, \infty), \bar p, q\in (2, \infty]$ such that $\frac{1}{2}=\frac{1}{p}+\frac{1}{\bar p}=\frac{1}{q}+\frac{1}{\bar q}\cdotp$ 
Remember from~\eqref{regu-eq3} that  $G(u)=F(u)-F(0)$.
 
 \smallskip
For any $0<\beta_1\leq \alpha,$ we use H\"{o}lder's inequality and Gagliardo-Nirenberg~\eqref{GN-ineq}:
\begin{align}
\left| \int \partial^\alpha u \times \partial^{\beta_1}(F(u)) \times |\nabla|^s_\delta\partial^{\alpha-\beta_1} u \right| 
&= \left| \int \partial^\alpha u \times \partial^{\beta_1}(G(u)) \times |\nabla|^s_\delta\partial^{\alpha-\beta_1} u\right| \notag\\
\leq&~ \|\partial^\alpha u\|_{L^2}\|\partial^{\beta_1}(G(u))\|_{L^{\frac{2(m-1)}{|\beta_1|-1}}}\||\nabla|^s_\delta\partial^{\alpha-\beta_1} u\|_{L^{\frac{2(m-1)}{m-|\beta_1|}}}\notag\\
\lesssim&~ \|u\|_{H^m} \|\nabla(G(u))\|_{L^\infty}^{1-\frac{|\beta_1|-1}{m-1}}\|\nabla (G(u))\|_{H^{m-1}}^{\frac{|\beta_1|-1}{m-1}}\|| \nabla|^s_\delta u\|_{L^{\infty}}^{1-\frac{m-|\beta_1|}{m-1}}\|| \nabla|^s_\delta u\|_{H^{m-1}}^{\frac{m-|\beta_1|}{m-1}}\notag\\
\lesssim&~  \|u\|_{H^m} \|G(u)\|_{H^m}^{1-\frac{|\beta_1|-1}{m-1}}\|G(u)\|_{H^{m}}^{\frac{|\beta_1|-1}{m-1}}\|u\|_{H^m}^{1-\frac{m-|\beta_1|}{m-1}}\|  u\|_{H^{m}}^{\frac{m-|\beta_1|}{m-1}}\lesssim\|u\|_{H^m}^{3}\label{estimate-2}.
\end{align}
For the second term in the expression \eqref{energy-es}, one uses the  estimate \eqref{estimate-frac} 
and Kato-Ponce inequality \eqref{KP-ineq1}; we have for any $0<\beta_2<\alpha$:
\[
\left|\int \partial^\alpha u \times |\nabla|^s_\delta\Bigl(\partial^{\beta_2} (F(u))\partial^{\alpha-\beta_2}u\Bigr)\right|
=\left|\int \partial^\alpha u|\nabla|^s_\delta\Bigl(\partial^{\beta_2} (G(u))\partial^{\alpha-\beta_2}u\Bigr)\right|
\leq \|\partial^\alpha u\|_{L^2}\|\partial^{\beta_2} (G(u))\partial^{\alpha-\beta_2}u\|_{\dot{H}^s}\notag\\
\]
\begin{equation}
\lesssim\|u\|_{H^m} \|\partial^{\beta_2} |\nabla|^s (G(u))\|_{L^{\frac{2(m-1)}{|\beta_2|}}}\|\partial^{\alpha-\beta_2}u\|_{L^{\frac{2(m-1)}{m-|\beta_2|-1}}}
+\|u\|_{H^m} \|\partial^{\beta_2} (G(u))\|_{L^\frac{2(m-1)}{|\beta_2|-1}}\|\partial^{\alpha-\beta_2}|\nabla|^s u\|_{L^\frac{2(m-1)}{m-|\beta_2|}}\label{es-beta2}.
\end{equation}
Using Gagliardo-Nirenberg inequality \eqref{GN-ineq} and Lemma \ref{compo-lemma} once more, we get
\begin{align*}
\|\partial^{\beta_2} |\nabla|^s (G(u))\|_{L^{\frac{2(m-1)}{|\beta_2|}}}\|\partial^{\alpha-\beta_2}u\|_{L^{\frac{2(m-1)}{m-|\beta_2|-1}}}
&\lesssim~\| |\nabla|^s (G(u))\|_{L^\infty}^{1-\frac{|\beta_2|}{m-1}}\| |\nabla|^s (G(u))\|_{H^{m-1}}^{\frac{|\beta_2|}{m-1}}\|\nabla u\|_{L^\infty}^{1-\frac{m-|\beta_2|-1}{m-1}}\|\nabla u\|_{H^{m-1}}^{\frac{m-|\beta_2|-1}{m-1}}\notag\\
&\lesssim~\| G(u)\|_{H^{m-1+s}}^{1-\frac{|\beta_2|}{m-1}}\| G(u)\|_{H^{m-1+s}}^{\frac{|\beta_2|}{m-1}}\| u\|_{H^m}^{1-\frac{m-|\beta_2|-1}{m-1}}\|u\|_{H^{m}}^{\frac{m-|\beta_2|-1}{m-1}}\notag\\
&\lesssim~ \| u\|_{H^{m}}^{2}
\end{align*}
and we estimate $\|\partial^{\beta_2} (G(u))\|_{L^\frac{2(m-1)}{|\beta_2|-1}}\|\partial^{\alpha-\beta_2}|\nabla|^s u\|_{L^\frac{2(m-1)}{m-|\beta_2|}}$ in the same way as we did in \eqref{estimate-2}.
Hence 
\begin{align}
\left|\int \partial^\alpha u \times |\nabla|^s_\delta\Bigl(\partial^{\beta_2} (F(u))\partial^{\alpha-\beta_2}u\Bigr)\right|
\lesssim \|u\|_{H^m}^3\label{estimate-3}.
\end{align}
The remaining term   in  the expression \eqref{energy-es}  is new  comparing to \cite{Im1}.
We take  advantage of the following   commutator estimate developed   recently due to Li \cite[eq. 1.8]{Li19}:
for any $s\in(0, 1] and 1<p<\infty,$
\begin{equation}\label{KP-ineq2}
\||\nabla|^s (fg)-f|\nabla|^s g\|_{L^{p}}\lesssim \||\nabla|^s f\|_{L^p} \|g\|_{L^{\infty}}.
\end{equation}
For any $0<\beta_3\leq\alpha-e_j$ we have, thanks to \eqref{KP-ineq2} with $p=2$:
\begin{align}
\Big|\int \partial^\alpha u &\times |\nabla|^s_\delta\Bigl(u\,\partial^{\beta_3} (F'(u))\partial^{\alpha-\beta_3}u\Bigr)\Big|
\leq~\|\partial^\alpha u\|_{L^2}\|u\, \partial^{\beta_3} (F'(u))\partial^{\alpha-\beta_3}u\|_{\dot{H}^s}\notag\\
&\lesssim \|u\|_{H^m}\left(\| |\nabla|^s u\|_{L^\infty} \|\partial^{\beta_3} (F'(u))\partial^{\alpha-\beta_3}u\|_{L^2}+\|u\|_{L^\infty} \|\partial^{\beta_3} (F'(u))\partial^{\alpha-\beta_3} u\|_{\dot{H}^s}\right)\label{es-beta33}.
\end{align}
The  term $\|\partial^{\beta_3} (F'(u))\partial^{\alpha-\beta_3} u\|_{\dot{H}^s}$ can be estimated similarly as we did in \eqref{es-beta2},
\textit{i.e.} we have
\begin{align}
&\|\partial^{\beta_3} (F'(u))\partial^{\alpha-\beta_3} u\|_{\dot{H}^s}\lesssim \|u\|_{H^m}^2.\label{es-beta3}
\end{align}
Notice that we could estimate $\|\partial^{\beta_3} (F'(u))\partial^{\alpha-\beta_3}u\|_{L^2}$ by  simply taking $s=0$ in \eqref{es-beta3}.
However, in order to  obtain a Beale-Kato-Majda blow-up criterion in the next subsection  (see Theorem \ref{Th-BKM}),
we shall now prove a sharper bound $\|u\|_{H^m}$ for this term, instead of a quadratic one.
Indeed, thanks to the Gagliardo-Nirenberg inequality \eqref{GN-ineq} and the fact that  $u$ is uniformly bounded, we have
\begin{align}
\|\partial^{\beta_3} (F'(u))\partial^{\alpha-\beta_3}u\|_{L^2}&\lesssim \|\partial^{\beta_3} \bigl(F'(u)-F'(0)\bigr)\|_{L^{\frac{2m}{|\beta_3|}}}\|\partial^{\alpha-\beta_3}u\|_{L^{\frac{2m}{m-|\beta_3|}}}\notag\\
&\lesssim \|F'(u)\|_{L^\infty}^{1-\frac{|\beta_3|}{m}}\|F'(u)\|_{H^m}^{\frac{|\beta_3|}{m}}\|u\|_{L^\infty}^{1-\frac{m-|\beta_3|}{m}}\|u\|_{H^m}^{\frac{m-|\beta_3|}{m}} 
\lesssim \|u\|_{H^m}.\label{BKM-use}
\end{align}

Putting \eqref{estimate-111}, \eqref{estimate-2}, \eqref{estimate-3}, \eqref{es-beta3} and \eqref{BKM-use} into \eqref{energy-es},
we obtain a differential inequality:
\begin{equation}
\frac{d}{dt}\|u\|_{H^m}^2\leq C \|u\|_{H^m}^{3}
\end{equation}
with a positive constant  $C$ independent of $\delta$ and of $\|u\|_{H^m}$. 
The continuation theorem for autonomous ordinary differential equations on a Banach Spaces then
ensures that the solution $u$ obtained in Step 1 can be extended to a time  $T$ which is independent
of $\delta$ as well. Namely, there exists a   time $T^*=(C\|u_0\|_{H^m})^{-1}$ such that
\begin{equation}\label{uniform-bound}
\forall t\in[0, T^*),\qquad
\|u(t)\|_{H^m}\leq \frac{\|u_0\|_{H^m}}{1-Ct\|u_0\|_{H^m}}.
\end{equation}

\begin{remark}\sl
In the case of $\R^d$, this step requires $F'(0+)<\infty$ to ensure~\eqref{BKM-use};
for $\T^d$, one can avoid the issue entirely thanks to the minimum principle (uniform positivity) established in Step 2.
\end{remark}

\medbreak\noindent
\textit{Step 4: Convergence.} We now turn to the convergence issue.
 For each $\delta \in(0, 1],$ let $u_\delta$ be the solution to \eqref{regu-eq1} constructed in the previous steps,
stemming from the same initial data $u(0, x)=u_0(x)$. Thanks to \eqref{uniform-bound}, the family $u_\delta\in\mathcal{C}([0, T); H^m)$ is uniformly bounded in terms of $\delta$ for any fixed $T<T^*.$
Next, we estimate the right-hand side of \eqref{regu-eq1}. Recall that $G(u)=F(u)-F(0)$. We have
\[
\|[F(u_\delta), |\nabla|^s_\delta]u_\delta\|_{H^{m-1}}
= \|[G(u_\delta), |\nabla|^s_\delta]u_\delta\|_{H^{m-1}}
\lesssim \|G(u_\delta)\|_{H^{m-1}}\||\nabla|^s_\delta u_\delta\|_{H^{m-1}}+\||\nabla|^s_\delta(u_\delta G(u_\delta))\|_{H^{m-1}}\\
\]
and thus, as $0<s\leq 1$:
\[
\|[F(u_\delta), |\nabla|^s_\delta]u_\delta\|_{H^{m-1}}
\lesssim\|u_\delta\|_{H^{m}}^{2}.
\]
This shows that the family $\{\partial_t u_\delta\}$ is uniformly bounded in $\mathcal{C}([0, T); H^{m-1})$ with respect to $\delta.$  

One could now unfold the classical weak compactness method to  establish a limit for a subsequence.
Instead, we will show that the entire family $\{u_\delta \,;\, \delta>0\}$ is a Cauchy sequence in $\mathcal{C}([0, T]; L^2).$ 
To prove this claim, we first need to prove an estimate on the difference of operator $|\nabla|^s_\delta$ defined by~\eqref{def-frac-operator}. For any fixed values~$\delta, ~\epsilon\in(0, 1]$, we have (using $2s\leq 2 \leq m$ in the last step):
\begin{align}
\||\nabla|^s_\delta f -|\nabla|^s_\epsilon f\|_{L^2}
&= \left\|\Big(\frac{1}{\delta}-T^s_\delta\Big)f-\Big(\frac{1}{\epsilon}-T^s_\epsilon\Big)f \right\|_{L^2}
=\left\|\int^\delta_\epsilon \partial_\tau \left(\frac{1-e^{-\tau|\xi|^s}}{\tau}\right)\,d\tau \,\widehat f\right\|_{L^2}\notag\\
&=\left\|\int^\delta_\epsilon\frac{1-(1+\tau|\xi|^s)e^{-\tau|\xi|^s}}{\tau^2} \,d\tau \,\widehat f \right\|_{L^2}\notag\\
&\leq \|\int^\delta_\epsilon\frac{1}{\tau^2} \bigl(\frac{1}{2}(\tau|\xi|^s)^2\bigr)\,d\tau \,\widehat f\|_{L^2}
\qquad\text{because}\enspace 1-(1+\vartheta)e^{-\vartheta}\leq \frac{1}{2}\vartheta^2\notag\\
&\leq \frac{1}{2}|\delta-\epsilon|\,\||\xi|^{2s}\widehat f\|_{L^2}
\leq \frac{1}{2}|\delta-\epsilon|\,\|f\|_{H^m}\label{estimate-difference-operator}.
\end{align}
Writing the equation for the difference of two solutions, we obtain
\begin{align*}
\partial_t(u_\delta-u_\epsilon)
&= \bigl(G(u_{\delta})-G(u_{\epsilon})\bigr)|\nabla|^s_\delta  u_\delta+G(u_{\epsilon})(|\nabla|^s _\delta-|\nabla|^s _\epsilon)u_\delta+G(u_{\epsilon})|\nabla|^s_\epsilon(u_\delta - u_\epsilon)\\
&\quad -(|\nabla|^s _\delta-|\nabla|^s _\epsilon)(G(u_{\delta}) u_\delta)-|\nabla|^s _\epsilon\bigl((G(u_{\delta})-G(u_{\epsilon})) u_\delta\bigr)-|\nabla|^s _\epsilon\bigl(G(u_{\epsilon})(u_\delta-u_\epsilon)\bigl).
\end{align*}
Taking $L^2$ inner product with $u_\delta-u_\epsilon,$ we further obtain
\begin{align*}
\frac{1}{2}\frac{d}{dt}\|u_\delta-u_\epsilon\|_{L^2}^2
&=\int (u_\delta-u_\epsilon)\,\bigl(G(u_{\delta})-G(u_{\epsilon})\bigr)|\nabla|^s_\delta  u_\delta+\int (u_\delta-u_\epsilon)\,G(u_{\epsilon})(|\nabla|^s _\delta-|\nabla|^s _\epsilon)u_\delta\\
&+\int (u_\delta-u_\epsilon)\,\Bigl([G(u_\epsilon), |\nabla|^s_\epsilon](u_\delta-u_\epsilon)\Bigr)-\int (u_\delta-u_\epsilon)\,(|\nabla|^s _\delta-|\nabla|^s _\epsilon)(G(u_{\delta}) u_\delta)\\
&-\int (u_\delta-u_\epsilon)\,|\nabla|^s _\epsilon\Bigl((G(u_{\delta})-G(u_{\epsilon})) u_\delta\Bigr).
\end{align*}
The third term cancels out by virtue of~\eqref{property-IBP}.
For the last term, one uses the point-wise inequality \cite{CM2015},
which can also be recovered directly by combining~\eqref{id_rem} with the kernel representation \eqref{kernel_rep}:
\[
-f |\nabla|^s f \leq -\frac{1}{2}|\nabla|^s (f^2).
\]
One gets (recall that $u_\delta$ and $F'$ are positive)
\begin{align*}
-\int (u_\delta-u_\epsilon)\,|\nabla|^s _\epsilon\Bigl((G(u_{\delta})-G(u_{\epsilon})) u_\delta\Bigr)
&=-\int (u_\delta-u_\epsilon)\,|\nabla|^s _\epsilon\Bigl((u_\delta-u_\epsilon)\,u_\delta\int_0^1 F'\bigl((1-\lambda)u_\epsilon+\lambda u_{\delta}\bigr)\,d\lambda\Bigr)\\
\leq&
-\frac{1}{2}\int (u_\delta-u_\epsilon)^2\,|\nabla|^s _\epsilon\Bigl(u_\delta\int_0^1 F'\bigl((1-\lambda)u_\epsilon+\lambda u_{\delta}\bigr)\,d\lambda\Bigr),
\end{align*}
which is bounded by
\begin{equation}
C\|u_\delta-u_\epsilon\|_{L^2}^2\,
\left\| u_\delta\int_0^1 F'\bigl((1-\lambda)u_\epsilon+\lambda u_{\delta}\bigr)\,d\lambda \right\|_{H^m}
\lesssim \|u_\delta-u_\epsilon\|_{L^2}^2\label{estimate-1111}
\end{equation}
in view of the uniform bounds on $u_\delta, u_\epsilon$ in $H^m$
and a generalization of Lemma \ref{compo-lemma} with two variables (see Chap. 5 in \cite{TW}).
The remaining terms can be estimated by \eqref{estimate-difference-operator}:
\begin{align*}
&\left|\int (u_\delta-u_\epsilon)\,\bigl(G(u_{\delta})-G(u_{\epsilon})\bigr)|\nabla|^s_\delta  u_\delta\right|
+\left|\int (u_\delta-u_\epsilon)\,G(u_{\epsilon})(|\nabla|^s _\delta-|\nabla|^s _\epsilon)u_\delta\right|\\
&\qquad+\left|\int (u_\delta-u_\epsilon)\,(|\nabla|^s _\delta-|\nabla|^s _\epsilon)(G(u_{\delta}) u_\delta)\right|\\
&\lesssim \|u_\delta-u_\epsilon\|_{L^2}\Bigl(\|G(u_{\delta})-G(u_{\epsilon})\|_{L^2}\||\nabla|^s_\delta u_\delta\|_{L^\infty}+\|G(u_\epsilon)\|_{L^\infty}\|(|\nabla|^s _\delta-|\nabla|^s _\epsilon)u_\delta\|_{L^2}\\
&\qquad+\|(|\nabla|^s _\delta-|\nabla|^s _\epsilon)(G(u_{\delta}) u_\delta)\|_{L^2}\Bigr)\\
&\lesssim \|u_\delta-u_\epsilon\|_{L^2}\bigl(\|u_\delta-u_\epsilon\|_{L^2} +|\delta-\epsilon|\bigr).
\end{align*}
In the last line, we have freely used the uniform bounds for $u_\delta, u_\epsilon$ in $H^m$, the fact that $H^{m-1}\subset L^\infty$,
a bound for $F'$ over the range of $u$ and, once more, the identity
\[
G(u_{\delta})-G(u_{\epsilon}) = (u_\delta-u_\epsilon)\int_0^1 F'\bigl((1-\lambda)u_\epsilon+\lambda u_{\delta}\bigr) d\lambda.
\]
Overall, we thus get
\[
\frac{1}{2}\frac{d}{dt}\|u_\delta-u_\epsilon\|_{L^2}^2\leq C \bigl(\|u_\delta-u_\epsilon\|_{L^2}^2+|\delta-\epsilon| \|u_\delta-u_\epsilon\|_{L^2}\bigr),
\]
where $C$ depends only on the initial conditions and other absolute dimensional quantities, but not on $\delta, \epsilon.$ Given that the solutions start with the same initial data, Gr\"onwall's  lemma implies that
\begin{equation}
\|u_\delta(t)-u_\epsilon(t)\|_{L^2}\leq C|\delta-\epsilon|(e^{Ct}-1)
\end{equation}
for all $t<T$. This proves our claim.

As a consequence of the interpolation inequality 
$$\|f\|_{H^{m'}}\leq \|f\|_{L^2}^{1-\frac{m'}{m}}\|f\|_{H^m}^{\frac{m'}{m}},\quad0<m'<m$$
and the uniform bound for $u_\delta$ in $\mathcal{C}([0, T); H^m),$  
one can state that $u_\delta$ converges strongly to some $u$ in all $\mathcal{C} ([0, T ); H^{m'} )$ when $\frac{d}{2}+1<m'< m$ (here $m'$ does not need to be an integer).
Moreover, $\partial_t u_\delta$ converges distributionally to $\partial_t u$, and in view of
the uniform bound of $\partial_t u_\delta$ in $H^{m-1},$ it does so strongly in $H^{m'-1}.$ 
This shows that the limit $u\in\mathcal{C} ([0, T ); H^{m'} )\cap \mathcal{C}^1 ([0, T ); H^{m'-1} )$ and that it
solves \eqref{GNB} classically, with the initial condition $u_0$.
Uniqueness is guaranteed by performing estimates that are similar to the ones we just established in Step 4.
Note that for the solution $u$ that we constructed, the maximum principle and the positivity proved earlier for $u_\delta$ still hold,
either by repeating the same argument based on the positivity of the kernel, or by passing to the limit in $L^\infty$.

\begin{remark}\sl
Let us point out that a variant of this step cannot be used to control a sequence of approximate solutions associated
with successive regularizations of a non-smooth initial data  $u_0$. Indeed, in that case,  the constants involved would
cease to be uniform with respect to $\varepsilon$, $\delta$. See~\S\ref{weak-solu} for an alternative approach.
\end{remark}

\medbreak\noindent
\textit{Step 5: Continuity of the solution.} At last, we prove that the unique solution $u$ belongs
to $\mathcal{C}([0, T], H^m)\cap \mathcal{C}^1([0, T), H^{m-1})$. 
By virtue of the equation  it is sufficient to show that  $u\in \mathcal{C}([0, T), H^m).$
For that, we first show that $u \in\mathcal{ C}_w ([0, T); H^m)$, which is the space of weakly
continuous $H^m$-valued functions. 
In view of the uniform bounds of $\partial_t u_\delta$ in $\mathcal{C}([0, T); H^{m-1})$ 
and $u_\delta$ in $\mathcal{C}([0, T); H^m),$ we know that $u\in L^\infty(0, T; H^m)$ 
and $\partial_t u\in L^\infty(0, T; H^{m-1}),$ in particular $u$ is almost everywhere equal
to a continuous function from $[0,T]$ into~$H^{m-1}$.
Finally,  the density of $H^{-(m-1)}$  in $H^{-m}$ implies that $u$ is weakly continuous
from $[0,T]$ into $H^m.$ More precisely, let $\langle\phi , u\rangle, ~\phi\in H^m$ denote
the dual paring of $H^{-m},$  there exist $\psi\in H^{-(m-1)}$ arbitrary close to $\phi$ in the sense
of the $H^m$-norm and the decomposition
\begin{align*}
\langle\phi, u\rangle(t)=\langle\phi-\psi, u\rangle(t)+\langle\psi, u\rangle(t),
\end{align*}
then implies the  continuity of $\langle\phi, u\rangle(t)$  on $[0, T).$

From the fact that $u\in\mathcal{ C}_w ([0, T); H^m)$  we have $\liminf_{t\to 0+}\|u(t)\|_{H^m}\geq \|u_0\|_{H^m}.$ For fixed $t\in [0, T),$ as the sequence  $u_\delta(t)$ is uniformly bounded in $H^m,$ it also admits a subsequence that  converges weakly to $u(t)$ in $H^m$; thus we have $\|u(t)\|_{H^m}\leq\limsup_{\delta\to0}\|u_\delta(t)\|_{H^m}.$ Recalling \eqref{uniform-bound} we further obtain
\begin{align*}
\limsup_{t\to 0+}\|u(t)\|_{H^m}&\leq\limsup_{t\to 0+}\limsup_{\delta\to0} \|u_\delta(t)\|_{H^m}\\
&\leq \limsup_{t\to 0+}\frac{\|u_0\|_{H^m}}{1-Ct\|u_0\|_{H^m}}\leq \|u_0\|_{H^m}.
\end{align*}
 In particular, $\lim_{t\to 0+}\|u(t)\|_{H^m}=\|u_0\|_{H^m}.$ This gives us strong right-continuity at $t = 0$ and,
 as the equation is also translation invariant, for any later time. The left-continuity for later times is obtained
 in the same fashion if one replaces~\eqref{uniform-bound} by
\[
\forall t,t'\in[0, T^*),\qquad
\|u(t)\|_{H^m}\leq \frac{\|u(t')\|_{H^m}}{1-C|t-t'|\|u(t')\|_{H^m}}\cdotp
\]
We can now conclude that $u$ is continuous on $[0, T)$.
This completes the proof of Theorem~\ref{Th_LWP}.
\end{proof}

\begin{remark}\sl
It is not known whether the result of Theorem~\ref{Th_LWP} can be extended to the case $s\in(1,2)$.
Indeed, in Step 3 of the proof, we used in a crucial way that there is only one singular term in the $H^m$-energy estimate and
that this term cancels out because of the commutator structure, which leads to~\eqref{energy-es}. When~$s>1$,
the $3d$~terms of~\eqref{energy-es} that are similar to $|\nabla|^s_\delta\partial^{\alpha-\beta_j}u$
are of order $m+s-1>m$ when $|\alpha|=m$ and $|\beta_j|=1$. As a mass cancellation is not likely, nor the use of~\eqref{id_rem};
this means that the well-posedness in $H^s$ of (GNB)$_s$ is not clear when $s>1$ and may require a different approach.
\end{remark}

\subsection{A Beale-Kato-Majda criterion} \label{A Beale-Kato-Majda criterion}

We now state the classical BKM criterion for our model.
\begin{theorem}\label{Th-BKM}
For $m>\frac{d}{2}+1$, suppose 
$u\in \mathcal{C}([0, T); H^m(\Omega))\cap\mathcal{C}^1([0, T); H^{m-1}(\Omega))$
is a positive solution  of~\eqref{GNB}  such that
\begin{equation}\label{in-BKM}
\int_0^T \|\nabla u (t)\|_{L^\infty}\, dt<+\infty.
\end{equation}
Then $u$ can be extended beyond time $T$ in the same regularity class.
\end{theorem}
\begin{remark}\label{Re-BKM}\sl
We will see that
$\displaystyle \int_0^T \||\nabla| u (t)\|_{L^\infty}\, dt<+\infty$
is also a BKM criterion.
\end{remark}

\begin{proof}
The proof relies on 
 an available  log-Besov  interpolation inequality  (Lemma \ref{lemma-log})   in the appendix. The reader may also refer to the Appendix~\ref{LPB} for the definition of  Besov spaces  and their properties, such as interpolation inequalities and embeddings.
In fact, we shall prove the following stronger BKM criterion: 
\begin{equation}\label{in-BKM-Besov}
\int^T_0 \|u(t)\|_{\dot{B}^1_{\infty, \infty}}\,dt<+\infty.
\end{equation}
According to the Bernstein's inequalities in Proposition \ref{P-Bernstin} and the fact that $\Delta_j$ is a uniformly  bounded operator in terms of $j$ in any $L^p$ spaces ($p\in[0, \infty]$), we have $2^j\|\Delta_j   u\|_{L^\infty}\lesssim \|\Delta_j  \nabla u\|_{L^\infty}\lesssim \|\nabla u\|_{L^\infty}.$ Hence \eqref{in-BKM} implies \eqref{in-BKM-Besov}. Similarly, since the symbol of operator $|\nabla|$ is $|\xi|\sim 2^j$, we have  $2^j\|\Delta_j   u\|_{L^\infty}\lesssim \|\Delta_j  |\nabla u|\|_{L^\infty}\lesssim \||\nabla| u\|_{L^\infty},$ thus the condition in Remark \ref{Re-BKM} also implies \eqref{in-BKM-Besov}.
From now on, let us assume that \eqref{in-BKM-Besov} holds. We will prove that the solution will not blow-up at time~$T$. 

Performing exactly the same estimates as in \eqref{energy-es}, \eqref{estimate-111}, \eqref{estimate-2}, \eqref{estimate-3}, \eqref{es-beta3} with $|\nabla|^s$ instead of $|\nabla|^s_\delta$, we arrive at the following a priori bound (note how we specifically used \eqref{BKM-use} to get the term $\||\nabla|^s u\|_{L^\infty}$):
\begin{align*}
\frac{1}{2}\frac{d}{dt}\|\partial^\alpha u\|_{L^2}^2
&\lesssim~\|u\|_{H^m}^2 \Bigl(\||\nabla|^s\bigl(u F'(u)\bigr)\|_{L^\infty} +
\sum_{0<\beta_1\leq\alpha} \|\nabla G(u)\|_{L^\infty}^{1-\frac{|\beta_1|-1}{m-1}}\||\nabla|^s u\|_{L^\infty}^{1-\frac{m-|\beta_1|}{m-1}}\\
&+\sum_{0<\beta_2<\alpha}\||\nabla|^s G(u)\|_{L^\infty}^{1-\frac{|\beta_2|}{m-1}}\|\nabla u\|_{L^\infty}^{1-\frac{m-|\beta_2|-1}{m-1}}+\|\nabla G(u)\|_{L^\infty}^{1-\frac{|\beta_2|-1}{m-1}}\||\nabla|^s u\|_{L^\infty}^{1-\frac{m-|\beta_2|}{m-1}}\\
&+\sum_{0<\beta_3\leq \alpha-e_j}\||\nabla|^s u\|_{L^\infty}\|F'(u)\|_{L^\infty}^{1-\frac{|\beta_3|}{m}}\|u\|_{L^\infty}^{1-\frac{m-|\beta_3|}{m}}
+\|u\|_{L^\infty}\|\nabla F'(u)\|_{L^\infty}^{1-\frac{|\beta_3|-1}{m-1}}\||\nabla|^s u\|_{L^\infty}^{1-\frac{m-|\beta_3|-1}{m-1}}\\
&+\|u\|_{L^\infty}\||\nabla|^s F'(u)\|_{L^\infty}^{1-\frac{|\beta_3|}{m-1}}\|\nabla u\|_{L^\infty}^{1-\frac{m-|\beta_3|-1}{m-1}}\Bigr).
\end{align*}
Using successively that $u\in\mathcal{C}([0, T); H^m(\Omega))\cap\mathcal{C}^1([0, T); H^{m-1}(\Omega)),$  
the maximal principle $\|u\|_{L^\infty}\leq \|u_0\|_{L^\infty}$,
Young's inequality, the embedding $\dot{B}^{0}_{\infty, 1} \hookrightarrow  L^{\infty}$ and Proposition \ref{P_Besov},
we rewrite the previous inequality: 
\begin{align}
\frac{1}{2}\frac{d}{dt}\| u\|_{H^m}^2
&\lesssim\|u\|_{H^m}^2 \Bigl(\| u F'(u) \|_{\dot{B}^{s}_{\infty, 1}} + \|G(u)\|_{\dot{B}^{1}_{\infty, 1}}+\||u\|_{\dot{B}^{s}_{\infty, 1}}+\|G(u)\|_{\dot{B}^{s}_{\infty, 1}}+\|u\|_{\dot{B}^{1}_{\infty, 1}}\notag\\
&+\|G(u)\|_{\dot{B}^{1}_{\infty, 1}}+\|u\|_{\dot{B}^{s}_{\infty, 1}}
+\| u\|_{\dot{B}^s_{\infty, 1}}(\|F'(u_0)\|_{L^\infty}+\|u_0\|_{L^\infty})\notag\\
&+\|u_0\|_{L^\infty}(\|F'(u)-F'(0)\|_{\dot{B}^{1}_{\infty, 1}}+\|u\|_{\dot{B}^{s}_{\infty, 1}})\notag\\
&+\|u_0\|_{L^\infty}(\||F'(u)-F'(0)\|_{\dot{B}^{s}_{\infty, 1}}
+\|u\|_{\dot{B}^{1}_{\infty, 1}})\Bigr)\notag\\
\lesssim&~\|u\|_{H^m}^2 \Bigl(\| u F'(u) \|_{\dot{B}^{s}_{\infty, 1}} + \|G(u)\|_{\dot{B}^{1}_{\infty, 1}}+\||u\|_{\dot{B}^{s}_{\infty, 1}}+\|G(u)\|_{\dot{B}^{s}_{\infty, 1}}+\|u\|_{\dot{B}^{1}_{\infty, 1}}\notag\\
&+\|F'(u)-F'(0)\|_{\dot{B}^{s}_{\infty, 1}}+\|F'(u)-F'(0)\|_{\dot{B}^{1}_{\infty, 1 }}\Bigr)\notag\\
\lesssim&~\|u\|_{H^m}^2 \bigl(\|u\|_{\dot{B}^{s}_{\infty, 1}}+\|u\|_{\dot{B}^1_{\infty, 1}}\bigr)\label{BKM-en}.
\end{align}
In the last step we have used a  composition lemma  for homogeneous Besov spaces \cite[Theorem 2.61]{Ba11}.

Next, one can get from the interpolation inequality in Proposition \ref{P_Besov} and Young's inequality that
\begin{align}
\|u\|_{\dot{B}^{s}_{\infty, 1}}\lesssim \|u\|_{\dot{B}^{0}_{\infty, \infty}}^{1-s}\|u\|_{\dot{B}^{1}_{\infty, \infty}}^{s}\lesssim \|u_0\|_{L^\infty}^{1-s}\|u\|_{\dot{B}^{1}_{\infty, \infty}}^{s}\lesssim 1+\|u\|_{\dot{B}^{1}_{\infty, 1}}\label{B-es1}
\end{align}
Moreover, by taking $r=1,\, p=\infty, \, \theta_1=1, \,\theta_2=m-d/2-1$ in Lemma \ref{lemma-log}, we see  from the embedding $H^m\hookrightarrow \dot{B}^{m-d/2}_{\infty, \infty}$ that
\begin{align}
\|u\|_{\dot{B}^1_{\infty, 1}}&\lesssim  \|u\|_{\dot{B}^1_{\infty, \infty}}\Bigl(1+\log_2 \Bigl(\frac{ \|u\|_{\dot{B}^{0}_{\infty, \infty}} +\|u\|_{\dot{B}^{m-d/2}_{\infty, \infty}}}{ \|u\|_{\dot{B}^1_{\infty, \infty}}}\Bigr)\Bigr)\notag\\
&\lesssim \|u\|_{\dot{B}^1_{\infty, \infty}}\Bigl(1+\log_2 \bigl( \|u_0\|_{L^\infty} +\|u\|_{H^m}\bigr)\Bigr)-\|u\|_{\dot{B}^1_{\infty, \infty}}\log_2 \|u\|_{\dot{B}^1_{\infty, \infty}}\notag\\
&\lesssim\|u\|_{\dot{B}^1_{\infty, \infty}}\Bigl(1+\log_2 \bigl( \|u_0\|_{L^\infty} +\|u\|_{H^m}\bigr)\Bigr)+1,\label{B-es2}
\end{align}
thanks to inequality $-a\log_2 a\leq 2$ on $\mathbb{R}^+.$

Hence,  substituting \eqref{B-es1} and \eqref{B-es2} into \eqref{BKM-en}, we obtain the  following differential inequality (note that $\| u(t)\|_{H^m}$ will not vanish):
\[
\frac{d}{dt}\| u\|_{H^m}\lesssim \| u\|_{H^m}\Bigl(1+\|u\|_{\dot{B}^{1}_{\infty, \infty}}\bigl(1+\log_2\bigl( \|u_0\|_{L^\infty} +\|u\|_{H^m}\bigr)\bigr)\Bigr).
\]
Let us define $X(t):= \ln (\| u\|_{H^m}+\|u_0\|_{L^\infty})$. We further obtain
\begin{align*}
\frac{d}{dt} X(t)\lesssim 1+\|u\|_{\dot{B}^{1}_{\infty, \infty}}\Bigl(1+X(t)\Bigr),
\end{align*}
and thus
\[
\frac{d}{dt} \ln (1+X(t))\lesssim 1+\|u\|_{\dot{B}^{1}_{\infty, \infty}}.
\]
One finally gets a double-exponential estimate of the form:
\begin{equation}
\|u(T)\|_{H^m}\lesssim \|u_0\|_{H^m}\exp\left[\exp\Bigl(T+\int^T_0 \|u(t)\|_{\dot{B}^1_{\infty, \infty}}\,dt\Bigr)\right].
\end{equation}
Theorem \ref{Th-BKM} follows immediately.
\end{proof}

\subsection{Instant regularization implies  global existence }\label{Instant regularization}

In this subsection we  study the question of the global existence in the periodic case through the lens of regularity theory.
The model (GNB) is a rather rare example of an equation of hydrodynamic flavor for which this strategy is fully successful.

\medskip
Suppose $u_0(x)>0, x\in\mathbb{T}^d$ and $u_0\in H^m(\mathbb{T}^d).$ From local existence theory Theorem \ref{Th_LWP}, there exists a unique classical positive solution $u\in\mathcal{C}([0, T); H^m(\mathbb{T}^d))\cap \mathcal{C}^1([0, T); H^{m-1}(\mathbb{T}^d)$ on the torus $\mathbb{T}^d.$ Let $T^*$ be its maximal time of existence. We will show that $T^*=+\infty.$ Let us assume, on the contrary, that it is finite. Then we infer from the BKM criterion of Theorem \ref{Th-BKM} that 
\begin{align}
\int^{T^{*}}_0 \|\nabla u(t)\|_{L^\infty({\mathbb{T}^d})}\,dt= +\infty.\label{contradiction-BKM}
\end{align}
Let us point out that, in what follows, the key is not to prove smoothness because one already knows that $\nabla u(t) \in H^{m-1} \subset L^\infty$ for any $t\in [0,T^\ast)$; instead, the point is rather to get a uniform control of this norm up to time $T^\ast$,
in order to contradict~\eqref{contradiction-BKM}.

\medskip
We first apply the De Giorgi-Nash-Moser  regularization scheme to our model.
Indeed, the equation~\eqref{GNB-w}  on the energy density $w=u^2$
is exactly of the kind studied  by L.~Caffarelli, C.H.~Chan, A.~Vasseur in \cite{Caf2011}.  
\begin{theorem}[\cite{Caf2011}]\label{Th_caf}
Let $\omega$ be a weak solution of the evolution equations of the type
\begin{equation}\label{GNE}
\partial_t\omega = \int_{\mathbb{R}^d} (\omega(y)-\omega(x))K(t, x, y)\,dy.
\end{equation}
For $0 < \bar s < 2$ and $0 < \Lambda$, if the kernel $K$ satisfies the properties:
\begin{align}
\forall x\neq y, \qquad K(t, x, y) = K(t, y, x) \label{CCV1}\\
\frac{\Lambda^{-1}}{|x-y|^{d+\bar s}} \leq  {K}(t, x, y) \leq \frac{\Lambda}{|x-y|^{d+\bar s}}. \label{CCV2}
\end{align}
Then for every $t_0 > 0,$  one has $\omega\in  \mathcal{C}^\alpha((t_0, \infty) \times \mathbb{R}^d )$ for some $\alpha>0$. Moreover,
the value $\alpha$ and the H\"older norm of $\omega$ depend exclusively on $t_0, d, \|\omega_0\|_{L^2} ,$ and $\Lambda$.
\end{theorem} 
Since $u$ is a classical solution, the formal passage from the equation \eqref{GNB-integro-differential} on~$u$
to the equation~\eqref{GNB-w} on $w=u^2$ holds true. Moreover,
in virtue of  the Max/ Min principle, \textit{i.e.} $\bar{ u}(t):=\max_{x\in\mathbb{T}^d} u(t, x)$ is a strictly decreasing function of time $t$ while $\underline {u}(t):=\min_{x\in\mathbb{T}^d} u(t, x)$ is a strictly increasing function of time $t,$ 
one can thus rewrite
\[
\frac{2u(x)u(y)}{u(x)+u(y)}=\frac{2}{\frac{1}{u(x)}+\frac{1}{u(y)}}
\]
and find that
\[
\frac{2u(x)u(y)}{u(x)+u(y)}\int_0^1 F'\bigl((1-\lambda)u(x)+\lambda u(y)\bigr)\,d\lambda\\
\leq \bar {u}(t)\, \max_{a\in [\underline{u}(t), \,\bar{u}(t)]} F'(a)\,\leq\, \bar {u}(0)\, \max_{a\in [\underline{u}(0), \,\bar{u}(0)]} F'(a)
\]
and
\[
\frac{2u(x)u(y)}{u(x)+u(y)}\int_0^1 F'\bigl((1-\lambda)u(x)+\lambda u(y)\bigr)\,d\lambda\\
\geq
 \underline {u}(t)\, \min_{a\in [\underline{u}(t), \,\bar{u}(t)]} F'(a)\,\geq\, \underline {u}(0)\, \min_{a\in [\underline{u}(0), \,\bar{u}(0)]} F'(a).
\]
Thus, if one defines
\begin{align*}
\Lambda :=c_{d, s}\max\Big\{ \bar {u}(0)\, \max_{a\in [\underline{u}(0), \,\bar{u}(0)]} F'(a),\quad \frac{1}{\underline {u}(0)\, \min_{a\in [\underline{u}(0), \,\bar{u}(0)]} F'(a)}\Bigr\},
\end{align*}
the  active kernel $\mathcal{K}^s$ given by \eqref{kernal-w} is symmetric with respect to $(x, y)$ and
satisfies~\eqref{CCV2}.
 Hence, Theorem~\ref{Th_caf} applies verbatim to our periodic solutions of \eqref{GNB-w}.
For any $0< t_0 <T^*$, there exists an $\alpha_0 > 0$, which depends
only on $t_0, d, \Lambda, \|u_0\|_{L^\infty}$,
that allows for a uniform $\mathcal{C}^{\alpha_0}((t_0, T^*) \times \mathbb{T}^d )$ bound:
\begin{align*}
\|w\|_{\mathcal{C}^{\alpha_0}((t_0, T^*)\times\mathbb{T}^d)}\leq C(t_0, d,  \Lambda, \|u_0\|_{L^\infty}).
\end{align*}
In particular, we also have uniform $\mathcal{C}^{\alpha_0}$ regularity on $(t_0, T^*)\times\mathbb{T}^d\times\mathbb{T}^d$ for 
\begin{align}
 m(t, x, y):&= c_{d, s}\, \frac{2 u(x)u(y)}{u(x)+u(y)}\int_0^1 F'\bigl((1-\lambda)u(x)+\lambda u(y)\bigr)\,d\lambda\notag\\
 &=c_{d, s}\, \frac{2 \sqrt{w(x)}\sqrt{w(y)}}{\sqrt{w(x)}+\sqrt{w(y)}}\int_0^1 F'\Bigl((1-\lambda)\sqrt{w(x)}+\lambda \sqrt{w(y)}\Bigr)\,d\lambda\label{def-m}
\end{align} 
since the solution $u$, and thus $w$, is bounded away from zero.   
Let us underline that we will use again the specific connection between $w$ and $m$
and that a given regularity of $w$ (collected later in the proof) will always imply the same regularity of $m.$

\bigskip
After this initial gain of $\mathcal{C}^{\alpha_0}$ regularity in space-time, the next step is to prove Schauder estimates on equation \eqref{GNB-w} or for similar types of equations, to bootstrap to higher order estimates.
Remark that the lack of evenness of the kernel puts our model out of the range of immediate applicability  of  recent results concerning the regularity theory of  nonlinear integro-differential equations, such as Caffarelli-Silvestre~\cite{Caf09}, Lara-D\`{a}vila~\cite{La14}, Mikulevicius-Pragarauskas~\cite{Miku} and   Jin-Xiong~\cite{Jin2, Jin1}, Dong-Zhang~\cite{DZ19}. 
However, the Schauder estimate obtained in \cite{Im2} for a general class of linear integro-differential equations (without evenness assumption on the kernel) has been applied successfully to \eqref{GNB-w} in the case $s=1$ and $F'\equiv 1.$

\smallskip
We  have the following result (proved in the Appendix~\ref{SCH}) for a class of general equations
where the unevenness of the kernel $Q$ is compensated by a slightly better integrability near the diagonal $z=0$. 
 \begin{theorem}\label{Th-Schauder-estimate}
 Let $s_0\in (0, 1]$ and $s_0\leq s\leq 1.$ Suppose that, for some $\alpha>0$,
$\omega\in\mathcal{C}^{1+\alpha, (1+\alpha)s}((-6, 0]\times\mathbb{R}^d)$ is a solution of the linear integro-differential equation:
\begin{equation}\label{schauder-eq}
\partial_t \omega=\int_{\mathbb{R}^d} \bigl(\omega(t, x+z)-\omega(t, x)\bigr)\, L(t, x, z)\,dz
+\int_{\mathbb{R}^d} \bigl(\omega(t, x+z)-\omega(t, x)\bigr) \,Q(t, x, z)\,dy+ \phi(t, x)
\end{equation}
 with $\phi(t, x)\in \mathcal{C}^{\alpha, \alpha s}((-6, 0]\times\mathbb{R}^d).$ Suppose $L$ and $Q$ satisfy 
 for all $(t, x, z)$ or $(t_i,x_i,z_i)\in (-6, 0]\times\mathbb{R}^d\times\mathbb{R}^d$:
 \begin{equation}\label{L1}
  L(t, x, z)=L(t, x, -z)
 \end{equation}
 \begin{equation}\label{L2}
 \Lambda_1|z|^{-d-s}\leq L(t, x, z)\leq \Lambda_2 |z|^{-d-s}
 \end{equation}
  \begin{equation}
| L(t_1, x_1, z)-L(t_2, x_2, z)|\leq \Lambda_2(|t_1-t_2|^\alpha+ |x_1-x_2|^{\alpha s})|z|^{-d-s}\label{L3}
 \end{equation}
 and
 \begin{equation}\label{Q1}
 | Q(t, x, z)|\leq \Lambda_2 \min \{1, |z|^{\alpha s}\}|z|^{-d-s}
 \end{equation}
  \begin{equation}
| Q(t_1, x_1, z)-Q(t_2, x_2, z)|\leq \Lambda_2\min\{|t_1-t_2|^\alpha+ |x_1-x_2|^{\alpha s}, |z|^{\alpha s}\} |z|^{-d-s}\label{Q2}
 \end{equation}
 respectively.
 Then for every $\beta<\alpha,$ there exists $C>0$ depending only on $s_0, d, \Lambda_1, \Lambda_2, \alpha, \beta$ such that
 \begin{equation}\label{ineq-Schauder-estimate}
 \|\omega\|_{\mathcal{C}^{1+\beta, (1+\beta)s}((-1, 0]\times\mathbb{R}^d)}\leq C\bigl(\|\omega\|_{L^\infty((-5, 0]\times\mathbb{R}^d)}+\|\phi\|_{\mathcal{C}^{\beta, \beta s}((-5, 0]\times\mathbb{R}^d)}\bigr).
 \end{equation}
 \end{theorem}
 
 \begin{remark}\sl
 In what follows, Theorem~\ref{Th-Schauder-estimate} will be applied successively with different values of $\alpha$, possibly with a time-regularity index $\alpha>1$ while the spatial regularity index $\alpha s<1$. In that case,
the assumptions~\eqref{L3} and \eqref{Q2} should be understood in the sense of H\"older semi-norms \eqref{holder_seminorm}-\eqref{holder_seminorm2}, that is that if an exponent $\vartheta$ exceeds 1, the offending term in the left-hand side is replaced with a derivative $\partial^{[\vartheta]}$ while the exponent in the
right-hand side is reduced to $\vartheta-[\vartheta]$, where $[\vartheta]$ denotes the integer part of~$\vartheta$.
 \end{remark}
 
Taking this result for granted, we return to (GNB) and follow the idea of \cite{Im2}.
One considers the energy density $w=u^2$ of a smooth local solution of~\eqref{GNB},
\textit{e.g.}~the finite lived one introduced at the begining of this~\S\ref{Instant regularization}.
The function $w$ is a smooth solution to \eqref{GNB-w}.
We ``restore the evenness'' in $z=x-y$ by  rewriting equation \eqref{GNB-w} in the following way:
 \begin{equation}
 \partial_t w = \int_{\mathbb{R}^d} \Bigl(w(t, y)-w(t, x)\Bigr)\frac{m(t, x, x)}{|x-y|^{d+s}}\,dy
 + \int_{\mathbb{R}^d}  \Bigl(w(t, y)-w(t, x)\Bigr)\frac{m(t, x, y)-m(t, x, x)}{|x-y|^{d+s}}\,dy\label{GNB-w-2}
 \end{equation}
where $m(t, x, y)$ is defined by \eqref{def-m} and satisfies $\|m\|_{\mathcal{C}^{\alpha_0}((t_0, T^*))\times\mathbb{T}^d\times\mathbb{T}^d)}\leq C(d, \Lambda, \|u_0\|_{L^\infty}).$ In this form it is clear that the regularity of $w$ and
of $m$ are sufficient to make sense of both integrals in~\eqref{GNB-w-2}.  Define
\begin{equation}
L(t, x, z):=\frac{m(t, x, x)}{|z|^{d+s}}\quad{\rm{and}}\quad Q(t, x, z):=\frac{m(t, x, x+z)-m(t, x, x)}{|z|^{d+s}}\cdotp
\end{equation}
It is elementary to check that $L$ and $Q$ satisfy the assumptions of Theorem \ref{Th-Schauder-estimate}
with $\phi=0$ and $\alpha=\alpha_0$ (note that $\alpha_0 s \leq \alpha_0$ because $0<s\leq 1$), thus we have
$$\|w\|_{\mathcal{C}^{\alpha_1, \alpha_1 s}((t_0, T^*)\times\mathbb{T}^d)}\leq C( s, d, \Lambda, \alpha_0, \alpha_1)$$
for every $\alpha_1<1+\alpha_0$. If $\alpha_1 s<1$, then one can update the uniform estimates of $L$ and $Q$ with $\alpha=\alpha_1$ and apply Theorem~\ref{Th-Schauder-estimate} iteratively $k$ times to gain uniform bounds
in $\mathcal{C}^{\alpha_k, \alpha_k s}$ with $\alpha_k < k+\alpha_0$.
In particular, without loss of generality, we can assume that $\alpha_{k_0} s> 1$.
Then:
$$\|w\|_{{\rm{Lip}}_{t, x}((t_1, T^*)\times\mathbb{T}^d)}\leq \|w\|_{\mathcal{C}^{\alpha_{k_0}, \alpha_{k_0} s}((t_1, T^*)\times\mathbb{T}^d)}\leq C( s, d, \Lambda, \alpha_0, \ldots,\alpha_{k_0}, k_0).$$
This is a contradiction to \eqref{contradiction-BKM}. Therefore, we now conclude that $T^*=\infty.$

\bigskip
Next, we investigate how to bootstrap across integer order of derivatives, \textit{i.e.}
we prove the high regularity estimates \eqref{higher-Holder-estimates-1} for $w$.
Note that, as $u$ remains bounded away from zero, the same estimates are also valid for $u.$
Differentiating \eqref{GNB-w-2} in $x$, we have for $w_1:=\nabla_x w,$
\begin{align*}
\partial_t w_1 &={\rm{p.v. }}\int_{\mathbb{R}^d}\bigl( w_1(t, x+z)-w_1(t, x)\bigr) \frac{m(t, x, x)}{|z|^{d+s}}\,dz\\
&\qquad +{\rm{p.v. }}\int_{\mathbb{R}^d}\bigl( w_1(t, x+z)-w_1(t, x)\bigr) \frac{m(t, x, x+z)-m(t, x, x)}{|z|^{d+s}}\,dz\\
&\qquad +{\rm{p.v. }}\int_{\mathbb{R}^d}\bigl( w(t, x+z)-w(t, x)\bigr) \frac{2\nabla_x m(t, x, x)}{|z|^{d+s}}\,dz\\
&\qquad +{\rm{p.v. }}\int_{\mathbb{R}^d}\bigl( w(t, x+z)-w(t, x)\bigr) \frac{2\nabla_x m(t, x, x+z)-2\nabla_x m(t, x, x)}{|z|^{d+s}}\,dz\\
&= \textrm{I}+\textrm{II}+\textrm{III}+\textrm{IV}.
\end{align*}
We used the property $m(t, x, y)=m(t, y, x)$ to simplify the expressions.
Recall that $w$ and $m$ are  now uniformly bounded in $\mathcal{C}^{\alpha, \alpha s}((t_0, T^*)\times\mathbb{T}^d\times\mathbb{T}^d )$
with $\alpha=\alpha_{k_0}$ and $\alpha s>1$.
By definition of fractional derivatives,
\[
\textrm{III}=2|\nabla|^s w \times \nabla_x m(t, x, x).
\]
Thus, as the map $w\mapsto |\nabla|^s w$ is bounded  from $\mathcal{C}^{\vartheta+s}$  into  $\mathcal{C}^{\vartheta}$:
\begin{align*}
\|\textrm{III}\|_{\mathcal{C}^{\beta, \beta s}((t_0, T^*)\times\mathbb{T}^d )}
&\leq  C\|w\|_{\mathcal{C}^{\beta, (\beta +1) s}((t_0, T^*)\times\mathbb{T}^d\times\mathbb{T}^d )}
\|m\|_{\mathcal{C}^{\beta, 1+\beta s}((t_0, T^*)\times\mathbb{T}^d\times\mathbb{T}^d )}\\
&\leq  C\|w\|_{\mathcal{C}^{\alpha, \alpha s}((t_0, T^*)\times\mathbb{T}^d\times\mathbb{T}^d )}^2\\
&\leq  C( s, d, \Lambda,\ldots , \alpha, \beta).
\end{align*}
for any $\beta < \beta_\ast$ with $\beta_\ast := \frac{\alpha s-1}{s} \in (0,\alpha)$.
Meanwhile, it follows from Proposition \ref{P-A2} that
$$\|\textrm{IV}\|_{\mathcal{C}^{\beta, \beta s}((t_0, T^*)\times\mathbb{T}^d )}\leq C \|w\|_{\mathcal{C}^{1, s}((t_0, T^*)\times\mathbb{T}^d\times\mathbb{T}^d )}\leq  C( s, d, \Lambda, \ldots, \alpha, \beta).$$
Applying Theorem \ref{Th-Schauder-estimate} to the equation of $w_1$ with, this time, $\phi = \textrm{III}+\textrm{IV}$, we thus obtain 
$$\|\nabla_x w\|_{\mathcal{C}^{1+\beta, (1+\beta)s}((t_0, T^*)\times\mathbb{T}^d)}
\leq C( s, d, \Lambda, \ldots, \alpha, \beta).$$
Similarly, we can  differentiate \eqref{GNB-w-2} in time $t$ and unfold a similar proof to obtain:
$$\|\partial_t w\|_{\mathcal{C}^{1+\beta,(1+\beta)s}((t_0, T^*)\times\mathbb{T}^d)}
\leq C( s, d, \Lambda, \ldots, \alpha, \beta).$$
The estimates \eqref{higher-Holder-estimates-1} of arbitrary order then follow from successive
differentiations of~\eqref{GNB-w-2} and applications of Theorem~\ref{Th-Schauder-estimate},
in  a fashion similar to the procedure that we just described. 

\medskip\break
We have now established the following result.

\begin{theorem}\label{Th_GWP}(Global regularity).
Given a pointwise \textbf{positive}  initial data $u_0$ in~$H^m(\mathbb{T}^d)$
for some integer $m>\frac{d}{2}+1$ and a non-local exponent $s\in(0,1]$,
the solution of problem \eqref{GNB}-\eqref{GNB-initialdata} obtained
in Theorem \ref{Th_LWP}  exists globally in time.
Furthermore, the solution is regularized instantly and satisfies the bounds \eqref{higher-Holder-estimates-1}.
\end{theorem}

\begin{remark}\sl
In the estimates, the constants can be chosen uniformly with respect to $s\in [s_0,1]$ for any $s_0>0$.
\end{remark}

In view of Theorem \ref{Th_GWP}, for smooth enough and positive initial data $u_0\in L^\infty(\mathbb{T}^d),$ the corresponding solution $u(t, x)$ is bounded from above and below and satisfies the higher order bounds \eqref{higher-Holder-estimates-1}, where all of these bounds depend only on the maximal and   minimal value of $u_0.$ 

\subsection{Global existence of weak solutions for positive data }\label{weak-solu}

To construct solutions stemming from a positive but not necessarily smooth bounded initial data, one needs a weak formulation
of the equation and strong a-priori bounds that will provide weak compactness to approximate solutions. Those bounds also
play a crucial role in ensuring the weak continuity of the solution at $t=0$. For subsequent times $t>0$, we shall have a classical
(smooth) solution in the limit.

\medskip
For example, when $s=1$ and $F(u)=u$ in \eqref{GNB}, which was the case for the (NB) equation considered in~\cite{Im1},
we did have "first momentum law" obtained by integrating \eqref{GNB}:
\[
\int_{\mathbb{T}^d}u(t', x)\,dx-\int_{\mathbb{T}^d}u(t, x)\,dx =\int_{t}^{t'}\int_{\mathbb{T}^d} u(\tau, x) \, |\nabla|u(\tau, x)\,dx\,d\tau
=\|u\|_{L^2(t, t'; \dot{H}^{1/2}(\mathbb{T}^d))}^2.
\]
This identity can be combined nicely with the energy conservation of the solutions and H\"older's embedding $L^2(\mathbb{T}^d)\subset L^1(\mathbb{T}^d)$ to ensure that $u\in L^2{(\mathbb{R}^+; \dot{H}^{1/2}(\mathbb{T}^d))},$ even regardless of the sign of $u_0$.
In the general case, the corresponding integral
\begin{equation}\label{momentum}
\frac{d}{dt}\int_{\mathbb{T}^d} u(t,x)dx = \int_{\mathbb{T}^d} F(u) |\nabla|^s u 
=\frac{1}{2} \iint_{\mathbb{T}^d\times \mathbb{T}^d} \left( F(u(x)) - F(u(y))\right) (u(x) -u(y)) K^s(x-y) dxdy
\end{equation}
remains signed because $F$ is assumed to be increasing on $\R^+$.
Formally, for smooth $u$ and $F$, this identity provides an $L^2{(\mathbb{R}^+; \dot{H}^{s/2}(\mathbb{T}^d))}$ control of $u$
because of the representation of the $\dot{H}^{s/2}(\mathbb{T}^d)$-norm with  finite differences:
\[
\|u\|_{L^2(\mathbb{R}^+; \dot{H}^{s/2}(\mathbb{T}^d))}^2
=\int_0^\infty \iint_{\mathbb{T}^d \times \mathbb{T}^d} \frac{|u(\tau, y)-u(\tau, x)|^2}{|x-y|^{d+s}}\,dx\,dy\,d\tau.
\]

An alternate path appears if one considers instead the evolution of the $L^p$ norms.
Indeed, if  $u(t, x)$ is a smooth 2$\pi$-periodic solution to our (GNB)  model \eqref{GNB}, then
\begin{align}
\|u(t, \cdot )\|_{L^p(\mathbb{T}^d)}^p\notag+ \frac{p}{2}\int_0^t \iint_{\mathbb{T}^d\times\mathbb{T}^d} &(|u(\tau, y)|^{p-2}-|u(\tau, x)|^{p-2})\notag\\
&\times \bigl(F(u(\tau, y))-F(u(\tau, x))\bigr)u(\tau, x)u(\tau, y)K^s_{{\rm{per}}}(x-y)\,dx\,dy\,d\tau\label{conservation-P}
\end{align}
is conserved for any $p\in (2, \infty).$ This property can be obtained by testing \eqref{GNB} with $|u|^{p-2}u.$
Instantly, by taking $p=3$ in this identity, we get
\begin{align*}
\frac{3}{2}\int_0^t \int_{\mathbb{T}^d}\int_{\mathbb{T}^d} \frac{|u(\tau, y)-u(\tau, x)|^2}{|y-x|^{d+s}} \mathcal{M}(\tau, x, y)\,dx\,dy\,d\tau\leq \|u_0\|_{L^3(\mathbb{T}^d)}^3
\end{align*}
where
\begin{align*}
\mathcal{M}(\tau, x, y):= u(\tau, x)u(\tau, y)\sum_{j\in\mathbb{Z}^d}\frac{c_{d, s}|x-y|^{d+s}}{|x-y+2\pi j|^{d+s}} \int_0^1 F'((1-\lambda) u(\tau, x)+\lambda u(\tau, y))\,d\lambda.
\end{align*}
Hence  we find that  for $s\in(0, 1]$:
\begin{equation}
\|u\|_{L^2(\mathbb{R}^+; \dot{H}^{s/2}(\mathbb{T}^d))}^2
\leq \frac{2}{3}\frac{1}{\min_{\tau, x, y}\mathcal{M}}\|u_0\|_{L^3(\mathbb{T}^d)}^3
\leq C_{d,s,F}\|1/u_0\|_{L^\infty(\mathbb{T}^d)}^2 \|u_0\|_{L^\infty(\mathbb{T}^d)}^3.\label{priori-H}
\end{equation}

We are now ready to construct weak solutions from arbitrary positive data in $L^\infty(\mathbb{T}^d).$  By global weak solutions of \eqref{GNB}, we mean that for any $\varphi\in\mathcal{C}^\infty(\mathbb{R}^+\times\mathbb{T}^d)$  the following weak formulation is satisfied
\begin{align}
\int_{\mathbb{T}^d} u(t, x)\varphi(t, x)\,dx &-\int_{\mathbb{T}^d} u_0(x)\varphi(0, x)\,dx -\int_{0}^t\int_{\mathbb{T}^d} u(\tau, x)\partial_t \varphi(\tau, x)\,dx\,d\tau\notag\\
&=\int_{0}^t\int_{\mathbb{T}^d}\int_{\mathbb{T}^d} \bigl( F(u(\tau, y))-F(u(\tau, x))\bigr)\varphi(\tau, x)u(y)\, {K}^s_{{\rm{per}}}(x-y)\,dx\,dy\,d\tau\label{weakform-u}
\end{align}
 for all $t>0$.
\begin{theorem}\label{Th_weaksolu}(Global weak solution)
Let $s\in (0, 1]$. For any initial data $u_0 \in L^\infty(\mathbb{T}^d), u_0 > 0$, there exists a global weak solution to \eqref{GNB} in the class
$$L^\infty(\mathbb{R}^+\times \mathbb{T}^d)\cap L^2(\mathbb{R}^+; \dot{H}^{s/2}(\mathbb{T}^d))\cap \mathcal{C}(\mathbb{R}^+; L^2(\mathbb{T}^d)).$$
The total energy $\|u(t,\cdot)\|_{L^2}^2$ and the quantity \eqref{conservation-P} are conserved,
the momentum $\int_{\mathbb{T}^d} u(t, x)\,dx$ is continuous on $\mathbb{R}^+$ and satisfies~\eqref{momentum}, \textit{i.e.}
\begin{align*}
\int_{\mathbb{T}^d}u(t', x)\,dx-\int_{\mathbb{T}^d}u(t, x)\,dx=\int_{t}^{t'}\int_{\mathbb{T}^d} F(u(\tau, x))|\nabla|^s u(\tau, x)\,dx\,d\tau
\geq 0.
\end{align*}
 Furthermore, for all $t > 0,$ $u$ satisfies the instant regularization estimates \eqref{higher-Holder-estimates-1} and the
 original (GNB) equation \eqref{GNB} is satisfied in the classical sense.
\end{theorem}
\begin{remark}\sl
 The continuity of the momentum at $t = 0$ prevents any concentration of the $\dot{H}^{s/2}$ norm in our weak solutions.
 If uniqueness was to fail, which is a possibility that one cannot rule out if $u_0$ is not smooth,
 the singular branching event could only occur at $t = 0.$ 
\end{remark}

\begin{proof}
In order to  prove the existence of weak solution in the class $L^\infty(\mathbb{R}^+\times \mathbb{T}^d)\cap L^2(\mathbb{R}^+; \dot{H}^{s/2}(\mathbb{T}^d)),$  one can resort to
the following classical procedure:
\begin{itemize}
\item[(1)] smooth out the positive and bounded initial data $u_0$ by taking standard mollifications of $u_0$ and get a sequence of global smooth solutions $(u_\epsilon)_{\epsilon>0}$ which satisfied the Max / Min principle and the regularization properties and thus \eqref{conservation-P}-\eqref{priori-H}; 
\item[(2)] prove that $(\partial_t u_\epsilon)_{\epsilon>0}$ is uniformly bounded in $L^2(\mathbb{R}^+\times\mathbb{T}^d)$ by using the commutator estimate \eqref{KP-ineq2}, then use  the Aubin-Lions lemma to get strong convergence of the sequence in $L^2(\mathbb{R}^+\times\mathbb{T}^d)$; 
\item[(3)]  finally,  show that $(u_\epsilon)_{\epsilon>0}$ converges, up to the extraction of a subsequence, to a solution $u$ of \eqref{GNB} in the sense of distributions.
\end{itemize}
Remark that, we have $u\geq \min_{x} u_0(x)>0.$ 
The only remaining problem is to restore the initial data and prove the time-continuity announced in the theorem. 
The key it to first prove the continuity of momentum. 
For any test function $\varphi(x)$ we rewrite \eqref{weakform-u} in a symmetric way:
\begin{align}
\int_{\mathbb{T}^d} u(t, x)\varphi(x)\,dx &-\int_{\mathbb{T}^d} u_0(x)\varphi(x)\,dx\notag\\
&=\frac{1}{2}\int_{0}^t\int_{\mathbb{T}^d\times\mathbb{T}^d} \bigl( F(u(\tau, y))-F(u(\tau, x))\bigr)(u(\tau, y)-u(\tau, x))\varphi(x)\, {K}^s_{{\rm{per}}}(x-y) \notag\\
&\qquad +\frac{1}{2}\int_{0}^t\int_{\mathbb{T}^d\times \mathbb{T}^d}\bigl( F(u(\tau, y))-F(u(\tau, x))\bigr)u(\tau, x)\bigl(\varphi(x)-\varphi(y)\bigr)\, {K}^s_{{\rm{per}}}(x-y).\label{weakform-u-2}
\end{align}
At this point, there are no a-priori bounds that guarantee the smallness of the first integral on the right-hand side when $t \to 0^+$. However,  we shall show that a possible concentration of the $\dot{H}^{s/2}$ norm near~$t = 0$ is not possible.
This goes back to an observation of the following lemma.
\begin{lemma}(Lemma 2.4 in \cite{Im1})
Suppose that a sequence of functions $\{u_n\} \subset L^\infty(\mathbb{T}^d)$, bounded
away from zero, enjoys both limits $u_n \rightharpoonup a $ and $u^2_n \rightharpoonup {b}^2$ in the weak$-\star$ topology of $L^\infty(\mathbb{T}^d)$. Then $b \geq a$.
\end{lemma}
Following the steps (1)-(3) outlined above, we know that there exists a weak solution $w=u^2$ that belongs to the class $L^\infty(\mathbb{R}^+\times \mathbb{T}^d)\cap L^2(\mathbb{R}^+; \dot{H}^{s/2}(\mathbb{T}^d))$ since $L^\infty\cap \dot{H}^{s/2}$ is an algebra,
and that satisfies the equation \eqref{GNB-w} in the weak sense, that is ($\mathcal{K}^s_{{\rm{per}}}$ represents the periodic version of $\mathcal{K}^s,$ which is symmetric  in terms of $x, y$):
\begin{align*}
\int_{\mathbb{T}^d} w(t, x)\varphi(t, x)\,dx &-\int_{\mathbb{T}^d} w(0, x)\varphi(0, x)\,dx -\int_{0}^t\int_{\mathbb{T}^d} w(\tau, x)\partial_t \varphi(\tau, x)\,dx\,d\tau\\
&=\frac{1}{2}\int_{0}^t\int_{\mathbb{T}^d}\int_{\mathbb{T}^d} \bigl( w(\tau, y)-w(\tau, x)\bigr)\bigl(\varphi(\tau, x)-\varphi(\tau, y)\bigr)\mathcal{K}^s_{{\rm{per}}}(\tau, x, y)\,dx\,dy\,d\tau.
\end{align*}
In particular, if we take $\varphi$ independent of $t,$ we find that $u^2(t)\rightharpoonup u^2_0$ weakly$-\star$ in $L^\infty(\mathbb{T}^d)$ as $t\to 0.$
Next, we  notice that for $\varphi(x)\geq0$ the first  integral of the right-hand side of \eqref{weakform-u-2} is signed:
\begin{align*}
&\int_{0}^t\int_{\mathbb{T}^d\times\mathbb{T}^d} \bigl( F(u(\tau, y))-F(u(\tau, x))\bigr)(u(\tau, y)-u(\tau, x))\varphi(x)\, {K}^s_{{\rm{per}}}(x-y)\geq 0.
\end{align*}
Meanwhile,  using H\"older's inequality and composition lemmas, 
 the second integral of the right-hand side of \eqref{weakform-u-2} is controlled by:
\begin{align*}
\Bigl|\int_{0}^t\int_{\mathbb{T}^d\times \mathbb{T}^d}\bigl( F(u(\tau, y)) &-F(u(\tau, x))\bigr)u(\tau, x)\bigl(\varphi(x)-\varphi(y)\bigr)\, {K}^s_{{\rm{per}}}(x- y)\Bigr|\\
&\leq C\sqrt{t}\,\|F(u)\|_{L^2(0, t; \dot{H}^{s/2})}\|\varphi\|_{\dot{H}^{s/2}}\|u\|_{L^\infty}\\
&\leq C\sqrt{t}\,\|u\|_{L^2(0, t; \dot{H}^{s/2})}\|\varphi\|_{\dot{H}^{s/2}}\|u\|_{L^\infty}\to 0\quad{\rm{as}}\quad t\to 0^+.
\end{align*}
Hence, any weak$-\star$ limit of a subsequence of $(u(t))_{t>0}$ would converge to a function $\tilde{u}$ satisfying 
$$\forall\,\,\varphi(x)\geq0, \qquad \int_{\mathbb{T}^d} (\tilde{u}-u_0)\,\varphi(x)\,dx \geq0.$$
Thus $\tilde{u}\geq u_0,$ which  combined with $u^2(t)\rightharpoonup u^2_0$ implies that $\lim_{t\to 0}u(t) =\tilde{u}=u_0$  in  the weak$-\star$  topology of $L^\infty(\mathbb{T}^d)$. In particular,  testing this weak$-\star$ limit with $\varphi\equiv 1$ ensures that the momentum $\int_{\mathbb{T}^d} u(t, x)\,dx$ is continuous at $t = 0.$ Lookinng back at \eqref{weakform-u-2}, it is now clear that
$$\|u\|_{L^2([0, t]; \dot{H}^{s/2}(\mathbb{T}^d))}\to 0 \quad{\rm{as}}\quad t\to 0^+.$$

Let us finally point out   that $u$ is weakly continuous in $L^2(\mathbb{T}^d)$ at $t=0$;
however, as $\|u(t, \cdot)\|_{L^2{(\mathbb{T}^d)}}$ is preserved and thererfore continuous at $t=0$,
the convervenge of $u(t,\cdot) \to u_0$ holds in the $L^2$ sense. Continuity at later times is not a problem since $u$ becomes
infinitely smooth.
\end{proof}

 In view of the time reversibility property mentioned in the introduction when $F$ is odd, if $u$
is a positive solution to \eqref{GNB}, then $-u(t^*-t)$ is a negative solution for any $t^*>0$. Thus
starting with positive data $u_0 \in L^\infty (\mathbb{T}^d)/\mathcal{C}(\mathbb{T}^d)$ we obtain a solution $u$ from Theorem \ref{Th_weaksolu} which becomes smooth instantaneously. Then $-u(t^*)$ serves as negative initial data that develop singularity at time $t=t^*.$
\begin{corollary}[Finite time singularity]
If $F$ is odd, for any $t^*>0$, there exists a negative initial condition $u_0 \in\mathcal{C}^\infty(\mathbb{T}^d)$ and there exists a classical solution to \eqref{GNB} on $[0,t^\ast]$ that develops into a discontinuous solution at time $t^*$ i.e. $u(t^*) \in L^\infty (\mathbb{T}^d)/\mathcal{C}(\mathbb{T}^d)$.
\end{corollary}

\section{Long-time asymptotics and stability}\label{sec3}

  As the solution is  squeezed by the maximum and minimum principles, it is expected that
  the long-time dynamics of the (GNB) model converges to a constant state consistent with the conservation of energy, namely,
\begin{equation}
u(t, x)\to \frac{\|u_0\|_{L^2(\mathbb{T}^d)}}{\sqrt{|\mathbb{T}^d|}}\quad{{\rm{as}}}\quad t\to +\infty.
\end{equation}
In this section, we first show that the amplitude of weak solutions tends to zero exponentially fast.
Then, we will exclude the persistence of high-frequency oscillations by showing that $|\nabla u|_{L^\infty}$
also tends to zero exponentially fast. 

\smallskip
Let us recall the notations:
$$\bar u(t) = \max_{x\in\mathbb{T}^d}u(t, x),\quad \underline{u}(t) = \min_{x\in\mathbb{T}^d}u(t, x)$$
and define the amplitude by 
\begin{equation}
A(t):= \bar u(t)-\underline{u}(t).
\end{equation}
\begin{theorem}[Large scale convergence]\label{decay-A}
Given  $u_0\in L^\infty(\mathbb{T}^d)$ with $u_0>0$ and a weak solution $u$ of~\eqref{GNB} associated with $u_0$ in the sense of~\eqref{weakform-u}, then $A(t)\leq A(0) e^{-\eta t}$ holds for all $t>0$ with some constant $\eta>0$ that depends only on  $d, s, \underline{u}(0)$ and
$\min\limits_{a\in[\underline{u}(0), \bar u(0)]} F'(a)$.
\end{theorem}
\begin{proof}
The proof is similar with our previous result \cite{Im2}, which relies on an idea from \cite{Im3}.
For a positive  initial data $u_0\in L^\infty(\mathbb{T}^d)$ we infer from Theorem \ref{Th_weaksolu} that $u$ is a global  weak solution, which  is smooth for all $t>0$.  Let us unfold  such solution   on $\mathbb{R}^d.$ There exist two points $\bar x, \underline{x}\in\mathbb{T}^d$ such that $\bar u(t)=u(t, \bar x)$ and $\underline{u}(t)=u(t, \underline{x}).$ The gradient $\nabla_x u$ vanish at both $\bar x, \underline{x}.$ We are going to evaluate \eqref{GNB-integro-differential} at $\bar x, \underline{x}.$
Using the fact that $F'\geq 0$ and the minimal principle, we have (we dropped the reference to time  for readability):
\begin{align*}
\frac{d}{dt}\bar u(t) &= \int_{\mathbb{R}^d} u(y)\bigl(F(u(y))-F(u(\bar x))\bigr)K^s(\bar x-y)\,dy\\
&\leq \underline{u}(0)\int_{|y-\bar x|\geq 1, | y-\underline{x}|\geq 1}\bigl(F(u(y))-F(u(\bar x))\bigr)K^s(\bar x-y)\,dy\\
&\leq \underline{u}(0)\int_{|y-\bar x|\geq 1, | y-\underline{x}|\geq 1}\bigl(F(u(y))-F(u(\bar x))\bigr)\min \{K^s(\bar x-y), K^s(\underline{x}-y)\}\,dy,
\end{align*}
and similarly
\begin{align*}
\frac{d}{dt}\underline{u}(t)=& \int_{\mathbb{R}^d} u(y)\bigl(F(u(y))-F(u(\underline{x}))\bigr)K^s(\underline{x}-y)\,dy\\
\geq&\, \underline{u}(0)\int_{|y-\bar x|\geq 1, | y-\underline{x}|\geq 1}\bigl(F(u(y))-F(u(\bar x))\bigr)K^s(\underline{x}-y)\,dy\\
\geq&\, \underline{u}(0)\int_{|y-\bar x|\geq 1, | y-\underline{x}|\geq 1}\bigl(F(u(y))-F(u(\underline{x}))\bigr)\min \{K^s(\bar x-y), K^s(\underline{x}-y)\}\,dy.
\end{align*}
Then mean value theorem  implies that
\begin{align*}
\frac{d}{dt} A(t)
&\leq -\underline{u}(0)\bigl(F(u(\bar x))-F(u(\underline{x}))\bigr)\int_{|y-\bar x|\geq 1, | y-\underline{x}|\geq 1}\min \{K^s(\bar x-y), K^s(\underline{x}-y)\}\,dy\\
&\leq -\underline{u}(0)\min_{a\in[\underline{u}(0), \bar u(0)]} F'(a)\, A(t)\int_{|y|\geq 1+|\bar x|+|\underline{x}|}\frac{c_{d, s}}{(|y|+|\bar x|+|\underline{x}|)^{d+s}}\,dy\\
&\leq -\eta A(t)
\end{align*}
 where $\eta=\underline{u}(0)\min\limits_{a\in[\underline{u}(0), \bar u(0)]} F'(a)\,\int_{|y|\geq 1+2\sqrt{d}\pi}\frac{c_{d, s}}{(|y|+2\sqrt{d}\pi)^{d+s}}\,dy.$
An application of Gr\"onwall's lemma completes the proof.
\end{proof}

\begin{theorem}[Small scale convergence]\label{Th-decay}
Given  $u_0\in L^\infty(\mathbb{T}^d)$ with $u_0>0$ and a weak solution $u$ of~\eqref{GNB} associated with $u_0$ in the sense of~\eqref{weakform-u}, then there exists a time $\widetilde T$ depending only on $s, d, \bar u(0), \underline{u}(0)$ and on
the extreme values of $F', F''$ on $[\underline{u}(0), \bar u(0)]$ such that
$\|\nabla u(t, \cdot)\|_{L^\infty}$ decay to zero exponentially fast starting from $t\geq \widetilde T$.
\end{theorem}
\begin{proof}
Let us unfold  $u$  on $\mathbb{R}^d$; the (GNB) equation~\eqref{GNB-integro-differential} can be rewritten as
\begin{align*}
\partial_t u  ={\rm{p.v. }}\int_{\mathbb{R}^d}\bigl(F(u(t, x+z))-F(u(t, x)\bigr)u(t, x+z)K^s(z)\,dz.
\end{align*}
After differentiating the equation and multiplying  by $\nabla u$ (the integrals being understood as principal values and we dropped the reference to time  for readability), one gets:
\begin{align}
\frac{1}{2}\partial_t |\nabla u(t, x)|^2 
&=\nabla u(x) \int_{\mathbb{R}^d}\nabla_x\bigl(F(u(x+z))-F(u( x)\bigr)u( x+z)K^s(z)\,dz\notag\\
&\qquad +\nabla u(x) \int_{\mathbb{R}^d} \bigl(F(u(x+z))-F(u(x)\bigr)\nabla_x u(x+z)K ^s(z)\,dz.\label{eq-gradient-u}\\
&= \nabla u(x) \int_{\mathbb{R}^d} F'(u(x))\bigl( \nabla_x u(x+z)-\nabla_x u(x)\bigr)u(x+z)K^s(z)\,dz\notag\\
&\qquad +\nabla u(x) \int_{\mathbb{R}^d} \bigl(F'(u(x+z))-F'(u(x))\bigr)\nabla_x u(x+z)u(x+z)K^s(z)\,dz\notag\\
&\qquad +\nabla u(x) \int_{\mathbb{R}^d} \bigl(F(u(x+z))-F(u(x)\bigr)\nabla_x u(x+z)K ^s(z)\,dz.\notag
\end{align}
If, from there on, $x\in\mathbb{T}^d$ is a point where the maximum value of $|\nabla u |$ is attained, one has
\begin{align*}
\nabla u(x) \int_{\mathbb{R}^d} F'(u(x))&\bigl( \nabla_x u(x+z)-\nabla_x u(x)\bigr)u(x+z)K^s(z)\,dz\notag\\
& =\frac{1}{2}\int_{\mathbb{R}^d} F'(u(x))\bigl(|\nabla_x u(x+z)|^2-|\nabla_x u(x)|^2\bigr)u(x+z)K^s(z)\,dz\notag\\
&\qquad-\frac{1}{2}\int_{\mathbb{R}^d} F'(u(x))|\nabla_x u(x+z)-\nabla_x u(x)|^2u(x+z)K^s(z)\,dz\notag\\
&\leq -  \frac{1}{2} \underline{u}(0) \times  \min_{a\in[\underline{u}(0), \bar u(0)]} F'(a)
\int_{\mathbb{R}^d} |\nabla_x u(x+z)-\nabla_x u(x)|^2K^s(z)\,dz,
\end{align*}
and~\eqref{eq-gradient-u} then takes the form:
\begin{align}
\frac{1}{2}\partial_t |\nabla u(t, x)|^2 &+
\frac{1}{2}\underline{u}(0) \times
\min_{a\in[\underline{u}(0), \bar u(0)]} F'(a)\int_{\mathbb{R}^d} |\nabla_x u(x+z)-\nabla_x u(x)|^2K^s(z)\,dz\notag\\
&\leq\nabla u(x) \int_{\mathbb{R}^d} \bigl(F(u(x+z))-F(u(x))\bigr)\nabla_x u(x+z) K^s(z)\,dz\notag\\
&\qquad+ \nabla u(x) \int_{\mathbb{R}^d} \bigl(F'(u(x+z))-F'(u(x))\bigr)\nabla_x u(x+z)u(x+z)K^s(z)\,dz\notag\\
&:= J_1+J_2.\label{es-gradient-u}
\end{align}
Meanwhile, we find  with an elementary identity followed by an integration by parts that
\begin{align*}
\int_{\mathbb{R}^d} |\nabla_x u(x+z)-\nabla_x u(x)|^2 K^s(z)\,dz
&\geq \,|\nabla_x u(x)|^2\int_{|z|\geq r}  K^s(z)\,dz-2\nabla_x u(x)\int_{|z|\geq r} \nabla_z \bigl(u(x+z) -u(x)\bigr) K^s(z)\,dz\\
&\geq \frac{C_1 }{r^s}|\nabla_x u(x)|^2+2\nabla_x u(x)\cdot\int_{|z|\geq r} \bigl(u(x+z) -u(x)\bigr) \nabla_z K^s(z)\,dz\\
&\qquad-2\nabla_x u(x)\cdot\int_{|z|= r}\nu_z(r)\bigl(u(x+z) -u(x)\bigr) K^s(z)\,d\sigma(r)
\end{align*}
where $\sigma(r)$ is the surface measure on $|z|=r$  and $\nu_z(r)$ is the outward-pointing normal vector to the sphere of
radius $r$ at a given point $z.$
It follows that if $|\nabla u (x) | = \| \nabla u \|_{L^\infty}$ then
\[
\int_{\mathbb{R}^d} |\nabla_x u(x+z)-\nabla_x u(x)|^2 K^s(z)\,dz\\
\geq \frac{C_1 }{r^s} |\nabla_x u(x)|^2-C_2|\nabla_x u(x)|\frac{A(t)}{r^{1+s}}
\]
and taking the optimal value $r=\frac{2 C_2  A(t)}{C_1 |\nabla u(x)|}$ gives
\begin{align}
\int_{\mathbb{R}^d} |\nabla_x u(x+z)-\nabla_x u(x)|^2 K^s(z)\,dz\geq \, 2C_3 \frac{\|\nabla  u \|_{L^\infty}^{2+s}}{A^s}
\cdotp\label{es-gradient-u-1}
\end{align}

To estimate $J_1,$ we shall split it depending on whether $|z|>\rho_1$ or $|z|\leq\rho_1$ and write $J_1=J_{11}+J_{12}.$ We estimate $J_{11}$ after rewriting it in the following form
\[
J_{11}=\frac{1}{2}\nabla u(x)\cdot \int_{|z|>\rho_1} \int_0^1 F'\bigl((1-\lambda)u(x)+\lambda u(x+z) \bigr)
\cdot\nabla_z |u(x+z)-u(x)|^2 K^s(z)\,d\lambda\,dz.
\]
Define $H(\lambda, t, x, z)=F'\bigl((1-\lambda)u(x)+\lambda u(x+z) \bigr)K^s(z)$. 
There exists $C_5>0$ such that
$$|H(\lambda, t, x, z)|\leq \max_{a\in[\underline{u}(0), \bar u(0)]} F'(a)\,K^s(z)\leq \frac{C_4}{|z|^{d+s}}$$
and
\begin{align*}
|\nabla_z H(\lambda, t, x, z)|
&\leq \max_{a\in[\underline{u}(0), \bar u(0)]} F'(a) \,|\nabla_z K^s(z)|+ \|\nabla u\|_{L^\infty}\,\max_{a\in[\underline{u}(0), \bar u(0)]} |F''(a)|\,  K^s(z)\\
&\leq  \frac{C_5}{|z|^{d+s+1}}+\|\nabla u\|_{L^\infty} \frac{C_5}{|z|^{d+s}}.
\end{align*}
Thus, with an integration by parts, we have for $J_{11}$:
\begin{align*}
|J_{11}|&\leq\frac{1}{2} |\nabla u(x)| \int_{|z|=\rho_1} \int_0^1 |u(x+z)-u(x)|^2 |H(\lambda, t, x, z)|\,d\lambda\,d\sigma(\rho_1)\\
&\,+\frac{1}{2}|\nabla u(x)| \int_{|z|>\rho_1}\int_0^1  |u(x+z)-u(x)|^2 |\nabla_z H(\lambda, t, x, z)|\,d\lambda\,dz\\
&\leq C_6\left(\|\nabla u\|_{L^\infty}\frac{A^2(t)}{\rho_1^{1+s}}+ \|\nabla u\|_{L^\infty}^2\frac{A^2(t)}{\rho_1^{s}}\right).
\end{align*}
As for $J_{12}$, we use the first order Taylor formula for the increment of $u:$
\begin{align*}
J_{12}&=\nabla u(x)\cdot  \int_{|z|\leq \rho_1} \int_0^1 F'(u(x+\lambda_1z))z\cdot \nabla_x u(x+\lambda_1 z)  \nabla_x u(x+z)K ^s(z)\,d\lambda_1\,dz\\
&= \nabla u(x)\cdot  \int_{|z|\leq \rho_1} \int_0^1 F'(u(x+\lambda_1z))z\cdot \big(\nabla_x u(x+\lambda_1 z)-\nabla_x u(x)\bigr) \nabla_x u(x+z)K ^s(z)\,d\lambda_1\,dz\\
&\quad+\nabla u(x)\cdot  \int_{|z|\leq \rho_1} \int_0^1  F'(u(x+\lambda_1z))z\cdot\nabla_x u(x) \bigl(\nabla_x u(x+z)-\nabla_x u(x)\bigr)K ^s(z)\,d\lambda_1\,dz\\
&\quad+|\nabla u(x)|^2\cdot  \int_{|z|\leq \rho_1} \int_0^1  F'(u(x+\lambda_1z))z\cdot\nabla_x u(x) K ^s(z)\,d\lambda_1\,dz:=J_{12}^1+J_{12}^2+J_{12}^3.
\end{align*}
We get, thanks to the  Cauchy-Schwarz  inequality, that
\begin{align*}
|J_{12}^1|
&\leq C_7  \|\nabla u\|_{L^\infty}^2  \int_{|z|\leq \lambda_1 \rho_1} \int_0^1 |\nabla_x u(x+  z)-\nabla_x u(x)| \sqrt{K^s(z)}\frac{(\lambda_1)^{s-1}}{|z|^{\frac{d+s}{2}-1}}\,d\lambda_1\,dz\\
&\leq  C_7  \|\nabla u\|_{L^\infty}^2  \Bigl(\int_{|z|\leq  \rho_1} |\nabla_x u(x+  z)-\nabla_x u(x)|^2 K^s(z)\,dz\Bigr)^\frac{1}{2} \rho_1^{1-s/2}\\
&\leq  \frac{1}{16}\min_{a\in[\underline{u}(0), \bar u(0)]} F'(a)\, \underline{u}(0)\int_{\mathbb{R}^d} |\nabla_x u(x+z)-\nabla_x u(x)|^2 K^s(z)\,dz\\
&\qquad+ \frac{16C_7^2}{\min\limits_{a\in[\underline{u}(0), \bar u(0)]} F'(a) \,\underline{u}(0)}\|\nabla u\|_{L^\infty}^4 \rho_1^{2-s}.
\end{align*}
The estimate for $J_{12}^2$ is completely analogous.
For $J_{12}^3,$ it is clear that 
\[
J_{12}^3\leq \max_{a\in[\underline{u}(0), \bar u(0)]} F'(a)  \|\nabla u\|_{L^\infty}^3 \int_{|z|\leq \rho} |z|K^s(z)\,dz
\leq C_8 \|\nabla u\|_{L^\infty}^3\,
\rho_1^{1-s}.
\]
Thus,
\begin{align*}
|J_{12}| \leq| J_{12}^1|+| J_{12}^2|+| J_{12}^2|
&\leq \frac{1}{8}\min_{a\in[\underline{u}(0), \bar u(0)]} F'(a) \,\underline{u}(0)\int_{\mathbb{R}^d} |\nabla_x u(x+z)-\nabla_x u(x)|^2 K^s(z)\,dz\\
&\qquad+ C_9\bigl(\|\nabla u\|_{L^\infty}^4 \rho_1^{2-s}+ \|\nabla u\|_{L^\infty}^3\,
\rho_1^{1-s}\bigr)
\end{align*}
and
\begin{align*}
|J_{1}|\leq |J_{11}|+|J_{12}|
&\leq \frac{1}{8}\min_{a\in[\underline{u}(0), \bar u(0)]} F'(a) \,\underline{u}(0)\int_{\mathbb{R}^d} |\nabla_x u(x+z)-\nabla_x u(x)|^2 K^s(z)\,dz\\
&\quad+ (C_6+C_9)\bigl(\|\nabla u\|_{L^\infty}^4 \rho_1^{2-s}+ \|\nabla u\|_{L^\infty}^3\,
\rho_1^{1-s}+\|\nabla u\|_{L^\infty}\frac{A^2(t)}{\rho_1^{1+s}}+ \|\nabla u\|_{L^\infty}^2\frac{A^2(t)}{\rho_1^{s}}\bigr).
\end{align*}
Choosing $\rho_1=A(t)/\|\nabla u\|_{L^\infty}$  gives 
\begin{align*}
|J_{1}|
\leq \frac{1}{8} \underline{u}(0) \times \min_{a\in[\underline{u}(0), \bar u(0)]} F'(a) \int_{\mathbb{R}^d} |\nabla_x u(x+z)-\nabla_x u(x)|^2 K^s(z)\,dz
+ 2(C_6+C_9)\|\nabla u\|_{L^\infty}^{2+s} (A^{1-s}(t)+A^{2-s}(t)).
\end{align*}
Now, we are going to estimate $J_2$ by rewriting it in a similar way as for $J_1$:
\begin{align*}
J_2=&\,\frac{1}{2}\nabla u(x) \int_{\mathbb{R}^d} \bigl(F'(u(x+z))-F'(u(x))\bigr)\nabla_z |u(x+z)|^2 K^s(z)\,dz\\
= &\,\frac{1}{2}\nabla u(x) \int_{|z|>\rho_2} \bigl(F'(u(x+z))-F'(u(x))\bigr)\nabla_z\bigl( |u(x+z)|^2 -|u(x)|^2\bigr)K^s(z)\,dz\\
&\,+\nabla u(x) \int_{|z|\leq\rho_2} \int_0^1 F''(u(x+\lambda_2 z))z\cdot \nabla_x u(x+\lambda_2 z)\cdot\nabla_x u(x+z)  u(x+z)  K^s(z)\,d\lambda_2\,dz\\
:=& J_{21}+J_{22}.
\end{align*}
One can now deal with $J_{21}$ in a similar fashion to what we did for $J_{11}$.  Indeed, we have
\begin{align*}
|J_{21}|\leq& \,\bar u(0)\|\nabla u\|_{L^\infty} A(t) \Bigl(\int_{|z|>\rho_2} \nabla_z\bigl(F'(u(x+z))K^s(z)\bigr)\,dz
+2\max_{a\in[\underline{u}(0), \bar u(0)]} F'(a)(\int_{|z|=\rho_2} K^s(z)\, d\sigma(\rho_2)\Bigr)\\
&\leq C_{10}\|\nabla u\|_{L^\infty} \left(\frac{A(t)}{\rho_{2}^{1+s}}+\|\nabla u\|_{L^\infty}\frac{A(t)}{\rho_{2}^{s}}\right).
\end{align*}
We split $J_{22}$ like we did for $J_{12}:$
\begin{align*}
|J_{22}|
&\leq\bar u(0) \|\nabla  u\|^2_{L^\infty}\|F''\|_{L^\infty_{{\rm{loc}}}(\mathbb{R}^+)}
\left(\int_{|z|\leq\rho_2}
\int_0^1 \bigl(\nabla_x u(x+\lambda_2 z)-\nabla_x u(x)\bigr) \cdot  |z|  K^s(z)\,d\lambda_2\,dz\right.\\
& \qquad \left. + \int_{|z|\leq\rho_2} \int_0^1  \bigl(\nabla_x u(x+  z)-\nabla_x u(x)\bigr) |z|  K^s(z)\,d\lambda_2\,dz
+\|\nabla u\|_{L^\infty}\int_{|z|\leq\rho_2} \int_0^1 |z|  K^s(z)\,d\lambda_2\,dz\right)\\
&\leq \frac{1}{8}\min_{a\in[\underline{u}(0), \bar u(0)]} F'(a) \,\underline{u}(0)\int_{\mathbb{R}^d} |\nabla_x u(x+z)-\nabla_x u(x)|^2 K^s(z)\,dz+ C_{11}\bigl(\|\nabla u\|_{L^\infty}^4 \rho_2^{2-s}+ \|\nabla u\|_{L^\infty}^3\,
\rho_2^{1-s}\bigr).
\end{align*}
Thus, we have
\begin{align*}
|J_2|\leq& \,\frac{1}{8}\min_{a\in[\underline{u}(0), \bar u(0)]} F'(a)\, \underline{u}(0)\int_{\mathbb{R}^d} |\nabla_x u(x+z)-\nabla_x u(x)|^2 K^s(z)\,dz\\
&+C_{12}\Bigl(\|\nabla u\|_{L^\infty}  \frac{A(t)}{\rho_{2}^{1+s}}+\|\nabla u\|^2_{L^\infty}\frac{A(t)}{\rho_{2}^{s}}+\|\nabla u\|^4_{L^\infty}\rho_{2}^{2-s}+\|\nabla u\|^3_{L^\infty}\rho_{2}^{1-s}\Bigr).
\end{align*}
Choosing $\rho_2=\sqrt{A(t)}/\|\nabla u\|_{L^\infty}$ gives
\begin{align*}
|J_2|\leq& \,\frac{1}{8}\min_{a\in[\underline{u}(0), \bar u(0)]} F'(a) \,\underline{u}(0)\int_{\mathbb{R}^d} |\nabla_x u(x+z)-\nabla_x u(x)|^2 K^s(z)\,dz
+C_{13}\|\nabla u\|^{2+s}_{L^\infty}(A^{\frac{1}{2}(1-s)}+A^{\frac{1}{2}(2-s)}).
\end{align*}
Substituting \eqref{es-gradient-u-1} into the left-hand side of \eqref{es-gradient-u} and with
the current estimates for $J_1, J_2$, we obtain:
\begin{align*}
\frac{d}{dt}  \|\nabla u\|_{L^\infty}^2 &+ C_3   \min_{a\in[\underline{u}(0), \bar u(0)]} F'(a) \,\underline{u}(0)   \frac{ \|\nabla u \|_{L^\infty}^{2+s}}{A^s(t)}\\
&\leq C_{14} \|\nabla u\|_{L^\infty}^{2+s}\bigl(A^{1-s}(t)+A^{2-s}(t)+A^{\frac{1}{2}(1-s)}(t)+A^{\frac{1}{2}(2-s)}(t)\bigr)\\
&\leq 4C_{14}  \|\nabla u \|_{L^\infty}^{2+s} \max\{A^{2-s}(t), A^{\frac{1}{2}(1-s)}(t)\}.
\end{align*}
In view of Theorem~\ref{decay-A}, there exists a time $T^\ast$ such that, for all $t\geq T^\ast$
\begin{equation}
4C_{14}\max\{A^{2-s}(t), A^{\frac{1}{2}(1-s)}(t)\}\leq \frac{1}{2} C_3   \min_{a\in[\underline{u}(0), \bar u(0)]} F'(a) \,\underline{u}(0).
\end{equation}
This implies that, for subsequent times:
\begin{align*}
\frac{d}{dt}  \|\nabla u\|_{L^\infty}^2+ \frac{1}{2}C_3   \min_{a\in[\underline{u}(0), \bar u(0)]} F'(a) \,\underline{u}(0)   \frac{\|\nabla u\|_{L^\infty}^{2+s}}{A^s(t)}\leq 0
\end{align*}
Using the precise estimate from Theorem~\ref{decay-A}, we further get
\begin{align*}
\frac{d}{dt}\|\nabla u\|_{L^\infty}^2+  \frac{e^{s\eta t}}{2C_3A^s(0)}  \min_{a\in[\underline{u}(0), \bar u(0)]}\, F'(a) \underline{u}(0) \|\nabla u\|_{L^\infty}^{2+s}\leq 0.
\end{align*}
This finally completes the proof of Theorem \ref{Th-decay}.
\end{proof}

\bigskip
Let us conclude this article with a stability result with respect to the nonlinearity $F$.
\begin{theorem}\label{thm-stab}
Let $F_1, F_2$ be two functions that satisfy our assumptions on the function $F.$ Given  two pointwise positive    initial data $u_{1, 0}, u_{2, 0}\in H^m(\mathbb{T}^d)$, we suppose that $u_i$ ($i=1,2$) are, respectively, the solution of the Cauchy problem \eqref{GNB}-\eqref{GNB-initialdata} with a nonlinearity $F_i$ and initial data $u_{i, 0}$.
\begin{enumerate}
\item We have the following stability estimate in $L^2(\mathbb{T}^d)$:
\begin{equation}
\forall\, t>0, \qquad
\|u_1-u_2\|_{L^2}\leq \big(\|u_{1, 0}-u_{2, 0}\|_{L^2}+ \|F_1'-F_2' \|_{L^\infty}\bigr)e^{C_0t},
\end{equation}
with a constant $C_0$ that depends on $d, \| u_1 \|_{H^m}, \| u_2 \|_{H^m}.$
\item In $L^\infty(\mathbb{T}^d),$ we also have an estimate that is independent of $F_i$:
\begin{equation}
\|u_1-u_2\|_{L^\infty}\leq 2\sqrt{d}\pi\bigl(\|\nabla u_1(t)\|_{L^\infty}+\|\nabla u_2(t)\|_{L^\infty}\bigr)+
\frac{1}{\sqrt{ \mathbb{T}^d}}
\left|  \|u_{1, 0}\|_{L^2(\mathbb{T}^d)} - \|u_{2, 0}\|_{L^2(\mathbb{T}^d)}\right|.
\end{equation}
\end{enumerate}
\end{theorem}
\begin{remark}\sl
From Theorem \ref{Th-decay} we see that when $t>\widetilde T,$ the difference between two solutions is essentially controlled by  $\left|  \|u_{1, 0}\|_{L^2(\mathbb{T}^d)} - \|u_{2, 0}\|_{L^2(\mathbb{T}^d)}\right|$. This is consistent  with the long time asymptotics of (GNB).
On the other hand, for a given non-linearity $F$,
all solutions stemming from a fixed energy level (\textit{i.e.} the intersection of an $L^2$-sphere with $H^m$) will end up uniformly close
to one another; the time $\widetilde{T}$ will be common among all solutions that have common pointwise upper and lower bounds.
\end{remark}
\begin{proof}
From  Theorem \ref{Th_GWP} we know that $u_1, u_2\in\mathcal{C}(\mathbb{R}^+; H^m)$  are positive and bounded
by the maximal value of their respective initial data.
The equation of the difference of the two solutions is:
\begin{align*}
\partial_t (u_1-u_2)=[F_1(u_1)-F_2(u_2), |\nabla|^s]u_1+[F_2(u_2), |\nabla|^s](u_1-u_2).
\end{align*}
Takint the $L^2$-inner product of this equation with $u_1-u_2$, we obtain 
\[
\frac{1}{2}\frac{d}{dt}\|u_1-u_2\|_{L^2}^2\\
 =\int_{\mathbb{T}^d} (u_1 -u_2) (F_1(u_1)-F_2(u_2)) |\nabla|^s u_1
 - \int_{\mathbb{T}^d} (u_1 -u_2)|\nabla|^s \bigl(u_1(F_1(u_1)-F_2(u_2))\bigr).
\]
Next, we decomposition $F_1(u_1)-F_2(u_2)=F_1(u_1)-F_1(u_2)+F_1(u_2)-F_2(u_2)$.
As mentioned in the introduction, we can assume that $F_1(0)=0$ without loss of generality.
The  composition lemma (Corollary 2.66 in~\cite{Ba11}), implies that the first integral can be estimated by
\begin{align*}
\left|\int_{\mathbb{T}^d} (u_1 -u_2) (F_1(u_1)-F_2(u_2)) |\nabla|^s u_1\right|
&\leq \|u_1 -u_2\|_{L^2}(\|F_1(u_1)-F_1(u_2)\|_{L^2}+\|F_1(u_2)-F_2(u_2)\|_{L^2})\||\nabla|^s u_1\|_{L^\infty}\\
&\leq  C\|u_1 -u_2\|_{L^2}(\|u_1 -u_2\|_{L^2}+ \|F_1 -F_2 \|_{L^\infty})\|u_1\|_{ {H}^{m}}.
\end{align*}
We then  split the second integral as
\begin{align*}
-\int_{\mathbb{T}^d} (u_1 -u_2)|\nabla|^s \bigl(u_1(F_1(u_1)-F_2(u_2))\bigr)
&= -\int_{\mathbb{T}^d} (u_1 -u_2)|\nabla|^s \bigl(u_1(F_1(u_1)-F_1(u_2))\bigr)\\
&\qquad-\int_{\mathbb{T}^d} (u_1 -u_2)|\nabla|^s \bigl(u_1(F_1(u_2)-F_2(u_2))\bigr).
\end{align*}
One can estimate $-\int (u_1 -u_2)|\nabla|^s \bigl(u_1(F_1(u_1)-F_1(u_2))\bigr)$
in a similar way as \eqref{estimate-1111}. Indeed,
\begin{align*}
-\int_{\mathbb{T}^d} (u_1 -u_2)|\nabla|^s \bigl(u_1(F_1(u_1)-F_1(u_2))\bigr)
&\leq C \|u_1 -u_2\|_{L^2} ^2 \| u_1\int_0^1 F'_1\bigl((1-\lambda)u_1+\lambda u_{2}\bigr)\,d\lambda \|_{H^m}\\
&\leq C \|u_1 -u_2\|_{L^2}^2\| u_1\|_{H^m}\| u_1, u_2\|_{H^m}.
\end{align*}
Recallin the fact that $L^\infty\cap \dot{H}^s$ is an algebra and making use of interpolation inequalities, we have:
\begin{align*}
\left|\int (u_1 -u_2)|\nabla|^s \bigl(u_1(F_1(u_2)-F_2(u_2))\bigr)\right|
&\leq \|u_1 -u_2\|_{L^2} \||\nabla|^s \bigl(u_1(F_1(u_2)-F_2(u_2))\bigr)\|_{L^2} \\
&\leq C\|u_1 -u_2\|_{L^2} \Bigl(\|u_1\|_{L^\infty}\|F_1(u_2)-F_2(u_2)\|_{\dot{H}^s}+\|u_1\|_{\dot{H}^s}\|F_1-F_2 \|_{L^\infty}\Bigr)\\
&\leq C\|u_1 -u_2\|_{L^2}\|u_1\|_{H^m}\Bigl(\|\nabla \bigl(F_1(u_2)-F_2(u_2)\bigr)\|_{L^2}+\|F_1-F_2 \|_{L^\infty}\Bigr)\\
&\leq C\|u_1 -u_2\|_{L^2}\|u_1, u_2\|_{H^m}\Bigl(\|F_1-F_2 \|_{L^\infty}+\|F_1'-F_2' \|_{L^\infty}\Bigr)\\
&\leq C\|u_1 -u_2\|_{L^2}\|u_1, u_2\|_{H^m}^2 \|F_1'-F_2' \|_{L^\infty}.
\end{align*}
Combining the previous estimates, we obtain:
\begin{equation}
\frac{d}{dt}\|u_1-u_2\|_{L^2}\leq C_0 \Bigl(\|u_1-u_2\|_{L^2} +\|F_1'-F_2' \|_{L^\infty}\Bigr)
\end{equation}
and finish the first statement  with the help of  Gronwall's lemma.

\medskip
For the second stament, let us point out that for $i=1, 2$:
$$\underline u(t)\leq\frac{\|u_{i, 0}\|_{L^2(\mathbb{T}^d)} }{\sqrt{ \mathbb{T}^d}}=\frac{\|u_{i}(t)\|_{L^2(\mathbb{T}^d)} }{\sqrt{ \mathbb{T}^d}}\leq \bar u(t).$$
In particular, there exists $y_i\in\mathbb{T}^d$ such that
$u(t, y_i)=\frac{\|u_{i, 0}\|_{L^2(\mathbb{T}^d)} }{\sqrt{ \mathbb{T}^d}}\cdotp$
One can then simply compare the values of the $u_i$ function to those asymptotic values:
\begin{align*}
\|u_1(t) -u_2(t) \|_{L_x^\infty}
&\leq \|u_1(t, x)-u_{1}(t, y_1)\|_{L_x^\infty}+\|u_2(t, x)-u_{2}(t, y_2)\|_{L_x^\infty}
+\left|\frac{\|u_{1, 0}\|_{L^2(\mathbb{T}^d)} }{\sqrt{ \mathbb{T}^d}}-\frac{\|u_{2, 0}\|_{L^2(\mathbb{T}^d)} }{\sqrt{ \mathbb{T}^d}}\right|\\
&\leq 2\sqrt{d}\pi\bigl(\|\nabla u_1(t)\|_{L^\infty}+\|\nabla u_2(t)\|_{L^\infty}\bigr)
+\frac{1}{\sqrt{ \mathbb{T}^d}} \left| \|u_{1, 0}\|_{L^2(\mathbb{T}^d)} \|u_{2, 0}\|_{L^2(\mathbb{T}^d)} \right|.
\end{align*}
Note that this last inequality is valid regardless of wether the functions $F_1$ and $F_2$ coincide or not.
\end{proof}


\appendix
\section{Littlewood-Paley decomposition and Besov Spaces}\label{LPB}
For the convenience of the reader and to keep this article as self-contained as possible, we recall briefly the
theory of the Littlewood-Paley decomposition, the definition of  Besov spaces and some useful properties.
More details and proofs can be found, \textit{e.g.}~in the book \cite{Ba11}. 

\medskip
Let  $ {\varphi}\in\mathcal{D}(\mathcal{C})$ be a smooth function supported in the annulus   $\mathcal{C}=\{\xi\in\mathbb{R}^3, \,\frac34\leq |\xi|\leq \frac83\}$ and such that 
\begin{gather*}
\sum_{j\in\mathbb{Z}}  {\varphi}(2^{-j}\xi)=1,\quad \forall \xi\in\mathbb{R}^3    \backslash \{0\}.
\end{gather*}
For $u\in \mathcal{S}'(\mathbb{R}^3)$, the frequency localization operator  $\dot{\Delta}_j$ and $\dot{S}_j$ are defined by 
\begin{equation*}
 \forall j \in \mathbb{Z},\quad \dot{\Delta}_ju:={\varphi}(2^{-j} D)u \quad\text{and}\quad
 \dot{S}_j u:=\sum_{\ell \leq j-1 }\dot{\Delta}_{\ell} u.
\end{equation*}
We have the formal decomposition
\begin{equation*}
\forall\, u \in{\mathcal S}'_h({\mathbb R}^3):= \mathcal{S}'(\mathbb{R}^3)/{\mathscr{P}}[\mathbb{R}^3], \qquad
u=\sum_{j\in\mathbb{Z}}\dot{\triangle}_ju.
\end{equation*}
where $\mathscr{P}[\mathbb{R}^3]$ is the set of polynomials.  Moreover, the Littlewood-Paley decomposition satisfies the property of almost orthogonality:
\begin{equation*}
{\dot\Delta}_j{\dot\Delta}_k u=0, \quad {\rm{if}} \ |j-k|\geq 2, \quad \dot\Delta_j(S_{k-1}u \dot\Delta_k u)=0, \quad {\rm{if}} \ |j-k|\geq 5.
\end{equation*}
We now recall the definition of homogeneous Besov spaces.
\begin{definition}
	Let $s$ be a real number and $(p, r)$ be in $[1,\infty]^{2}$, we set
	\begin{equation*}
	\|u\|_{\dot B^{s}_{p, r}} :=  \left\{
	\begin{split}
	&\|2^{j s}\|\dot\Delta_{j}u\|_{L^{p}(\mathbb{R}^d)}\|_{\ell^r(\mathbb{Z})} \quad\, {\rm{for}} ~ 1 \leq r < \infty,\\
	& \sup_{j \in \mathbb{Z}} 2^{j s}\|\dot\Delta_{j}u\|_{L^{p}}~\,\,\quad\qquad {\rm{for}} ~  r= \infty.
	\end{split}
	\right.
	\end{equation*}
	The  corresponding homogeneous Besov space is defined by $\dot B^{ s}_{p, r}:= \{ u \in \mathcal{S}^{'}_{h}(\mathbb{R}^3),\,	\|u\|_{\dot B^{s}_{p, r}}<\infty \}$.
\end{definition}
\noindent
For example, it is clear that $\|\cdot\|_{\dot{H}^{s}}=\|\cdot\|_{\dot{B}^{s }_{2, 2}}.$
Moreover, we have ${B}^{s}_{p, r}\hookrightarrow \dot{B}^s_{p, r}$ whenever~$p$ is finite and~$s$ is positive.

\medskip
Next, we state some usefull facts about Littlewood-Paley theory and Besov spaces (see \cite{Ba11} for details).
Note that, in the following results, one can harmlessly replace $\nabla^k$ by $|\nabla|^k$ if necessary.
\begin{proposition}\label{P-Bernstin} 
	Fix some $0<r<R.$ 
	A constant $C$ exists such that for any nonnegative integer $k$, any couple $(p, q)$ in $[1, \infty]^2$ with $q\geq p\geq 1$ and any function $u$ of $L^p$ with 
	${\rm{Supp}}~\widehat u\subset \{\xi\in{\mathbb R}^d,\; |\xi|\leq \lambda R\},$ we have
	\begin{equation*}
	 \|\nabla^{k}u\|_{L^{q}}\leq C^{k+1}\lambda^{k+d(\frac{1}{p}-\frac{1}{q})}\|u\|_{L^{p}}.
	\end{equation*}
	If $u$ satisfies   ${\rm{Supp}}~\widehat u\subset \{\xi\in{\mathbb R}^d,\; r\lambda\leq |\xi|\leq R\lambda\},$  then we  have
	\begin{equation*}
	C^{-k-1}\lambda^{k}\|u\|_{L^{p}}\leq   \|\nabla^{k}u\|_{L^{p}} \leq C^{k+1}\lambda^{k}\|u\|_{L^{p}}.
	\end{equation*}
\end{proposition}
	
\begin{proposition}\label{P_Besov}
	Let $1\leq p \leq \infty$. Then there hold:
	
	\begin{itemize}
		\item for all $s\in\mathbb R$~~and~~$1\leq p, r\leq\infty,$ 
		we have 
			\begin{equation*}
		\|\nabla^{k}u\|_{\dot B^{ s}_{p, r}}\simeq\|u\|_{\dot B^{ s+k}_{p, r}}.
		\end{equation*}
		\item for any $\theta\in(0, 1)$ and $\,  \underline{s}<\bar  s,$ we have
$$\|u\|_{\dot B^{\theta   \underline{ s}+(1-\theta)\bar{s}}_{p, 1}}\lesssim\|u\|_{\dot B^{  \underline {s}}_{p, \infty}}^{\theta}\|u\|_{\dot B^{\bar{s}}_{p, \infty}}^{1-\theta}.$$
        \item Embedding: we have the following continuous embedding
$$\dot B^{s}_{p,   r}\hookrightarrow \dot B^{ s-\frac{d}{p}}_{\infty, \infty}\quad{\rm{ whenever }} \ 1\leq p,   r\leq\infty,$$
and
$$\dot{B}^0_{\infty, 1}\hookrightarrow L^\infty\hookrightarrow\dot{B}^0_{\infty, \infty}.$$
	\end{itemize}
\end{proposition}

\begin{lemma}\label{lemma-log}
Let $s\in\mathbb{R}.$ For all $\theta_1, \theta_2>0$ and $1\leq p\leq \infty,$
there exists a constant $C$ depends on $\theta_1, \theta_2$ such that
\begin{equation}\label{in-log}
\|f\|_{\dot{B}^{r}_{p, 1}}\leq C(\theta_1, \theta_2) \|f\|_{\dot{B}^r_{p, \infty}}
\left(
1+\log_2\left(\frac{\|f\|_{\dot{B}^{r-\theta_1}_{p, \infty}}+\|f\|_{\dot{B}^{r+\theta_2}_{p, \infty}}}{\|f\|_{\dot{B}^r_{p, \infty}}}\right)
\right).
\end{equation}
\end{lemma}
\begin{proof}
The proof is exactly the same as in the case  $\theta_1=\theta_2$, which is classic and can be found in~\cite{DaNOTE}. 
\end{proof}

\section{Schauder estimates \& proof of Theorem \ref{Th-Schauder-estimate}}\label{SCH}

We have to mention that for the case $s=1,$ Theorem \ref{Th-Schauder-estimate} has been proved in \cite{Im2}.
Here we give a proof for the general case $s\in(0, 1)$ which will relies on the following propositions and lemma:
\begin{proposition}\label{P-A1}
Let $\phi(t, x)\in \mathcal{C}^{\alpha, \alpha s}((-6, 0]\times\mathbb{R}^d)$
and $\omega\in \mathcal{C}^{1+\alpha, (1+\alpha) s}((-6, 0]\times\mathbb{R}^d)$ be a solution of the following integro-differential equation
(the equation in Theorem \ref{Th-Schauder-estimate} with $Q\equiv 0$):
\begin{align}
&\partial_t \omega=\int_{\mathbb{R}^d} \bigl(\omega(t, x+z)-\omega(t, x)\bigr)\, L(t, x, z)\,dz+ \phi(t, x).\label{eq-A1}
\end{align}
Suppose that $L$ satisfy \eqref{L1}, \eqref{L2} and \eqref{L3} with the same value of $\alpha>0$.
There exists a constant $C>0$ depending only on $s, d, \Lambda_1, \Lambda_2, \alpha$ such that
\[
\|\omega\|_{\mathcal{C}^{1+\beta, (1+\beta) s}((-2, 0]\times\mathbb{R}^d)}
\leq C(\|\omega\|_{L^\infty((-5, 0]\times \mathbb{R}^d)}+\|\phi\|_{\mathcal{C}^{\beta, \beta s}((-5, 0]\times\mathbb{R}^d)})
\]
for any $\beta\leq \alpha$.
\end{proposition}
\begin{proof}
As $L$ is even in $z$, this result is quite natural and we adapt the proof of \cite[Proposition 2.1]{Im2}, bootstrapping by
increments of $s$ in scale of spatial regularity.
At first, {we know from the H\"older estimates in~\cite{La14} }(see also \cite{Fe2013}) that there exist positive $\gamma$ and  $C$ depending only on $s, d, \Lambda_1, \Lambda_2$ such that 
\begin{align*}
\|\omega\|_{\mathcal{C}^{\gamma, \gamma s}((-4, 0]\times\mathbb{R}^d )}
\leq C(\|\omega\|_{L^\infty((-5, 0]\times\mathbb{R}^d)}
+\|\phi\|_
{L^\infty((-5, 0]\times\mathbb{R}^d)}
).
\end{align*}
Then we infer \textit{e.g.} from Theorem 1.1 in \cite{DZ19} that for $\gamma_1= 1+ (\gamma \wedge \alpha)$
and $\alpha_1=\gamma \wedge \alpha$:
\begin{align*}
\|\omega\|_{\mathcal{C}^{\gamma_1, \gamma_1 s}((-3, 0]\times\mathbb{R}^d )}
&\leq C(\|\omega\|_{\mathcal{C}^{\alpha_1, \alpha_1 s}((-4, 0]\times\mathbb{R}^d )}
+\|\phi\|_{\mathcal{C}^{\alpha_1, \alpha_1 s}((-5, 0]\times\mathbb{R}^d)})\\
&\leq C(\|\omega\|_{\mathcal{C}^{\gamma, \gamma s}((-4, 0]\times\mathbb{R}^d )}
+\|\phi\|_{\mathcal{C}^{\gamma, \gamma s}((-5, 0]\times\mathbb{R}^d)})\\
&\leq C(\|\omega\|_{L^\infty((-5, 0]\times\mathbb{R}^d)}+\|\phi\|_{\mathcal{C}^{\gamma, \gamma s}((-5, 0]\times\mathbb{R}^d)}).
\end{align*}
If $\gamma < \alpha$, which is generally expected, successive applications of this result provide a uniform control of each
norm $\mathcal{C}^{\gamma_k, \gamma_k s}((-3, 0]\times\mathbb{R}^d )$
with $\gamma_k = 1+((\gamma+k) \wedge \alpha)$; choosing $k$ large enough provides the result.
\end{proof}

\begin{proposition}\label{P-A2}
Suppose $Q$ satisfy \eqref{Q1} and \eqref{Q2}. Define 
\begin{equation*}
Q_\omega (t, x) := \int_{\mathbb{R}^d}\bigl( \omega(t, x+z)-\omega(t, x)\bigr) Q(t, x, z)\,dz.
\end{equation*}
For any $0<\beta<\alpha$,
there exists $C>0$ depending only on $s, d, \Lambda_2, \alpha, \beta$ such that
$$\|Q_\omega\|_{\mathcal{C}^{\beta, \beta s}((-5, 0]\times\mathbb{R}^d)}\leq C\|\omega\|_{\mathcal{C}^{1, s}((-5, 0]\times\mathbb{R}^d)}.$$
\end{proposition}
\begin{proof}
First, using \eqref{Q1}, we find that
\begin{align*}
\|Q_\omega \|_{L^\infty((-5, 0]\times\mathbb{R}^d)}
&\leq \|\omega\|_{\mathcal{C}^{s}_x((-5, 0]\times\mathbb{R}^d)}\int_{B_{1}(z)} |z|^{s}|Q|\,dz
+2\|\omega\|_{L^\infty((-5, 0]\times\mathbb{R}^d)}\int_{\mathbb{R}^d/B_{1}(z)} |Q|\,dz\notag\\
&\leq 2 \|\omega\|_{\mathcal{C}^{s}_x((-5, 0]\times\mathbb{R}^d)}\int_{\mathbb{R}^d}\min\{1, |z|^{s}\} |Q|\,dz\notag\\
&\leq C\|\omega\|_{\mathcal{C}^{s}_x((-5, 0]\times\mathbb{R}^d)}\int_{\mathbb{R}^d} \min\{1, |z|^{s}\}\min\{1, |z|^{\alpha s}\} |z|^{-d-s}\,dz\notag\\
&\leq  C \|\omega\|_{\mathcal{C}^{s}_x((-5, 0]\times\mathbb{R}^d)}.
\end{align*}
For any $(t, x), (\tau, \xi)\in (-5, 0]\times\mathbb{R}^d$ with $0<|x-\xi|\leq \frac{1}{e}.$  Similarly, using \eqref{Q1} and \eqref{Q2}, we have
\begin{align*}
|Q_\omega (t, x)-Q_\omega (t, \xi)|
&\leq \left|\int_{\mathbb{R}^d}\bigl( \omega(t, x+z)-\omega(t, x)\bigr) \bigl(Q(t, x, z)-Q(t, \xi, z)\bigr)\,dz\right|\\
&\qquad+\left|\int_{\mathbb{R}^d}\bigl( \omega(t, x+z)-\omega(t, \xi+z)+\omega(t, \xi)-\omega(t, x)\bigr) Q(t, \xi, z)\,dz\right|\\
&\leq 2\|\omega\|_{\mathcal{C}^{s}_x((-5, 0]\times\mathbb{R}^d)}\int_{\mathbb{R}^d}\min\{1, |z|^{s}\} |Q(t, x, z)-Q(t, \xi, z)|\,dz\notag\\
&\qquad+2\|\omega\|_{\mathcal{C}^{s}_x((-5, 0]\times\mathbb{R}^d)}\int_{\mathbb{R}^d}\min\{|x-\xi|^{s}, |z|^{s}\} |Q(t, \xi, z)|\,dz\notag\\
&\leq C\|\omega\|_{\mathcal{C}^{s}_x((-5, 0]\times\mathbb{R}^d)}\Bigl(\int_{\mathbb{R}^d}\min\{1, |z|^{s}\} \min\{|x-\xi|^{\alpha s}, |z|^{\alpha s}\} |z|^{-d-s}\,dz\\
&\qquad+\int_{\mathbb{R}^d} \min\{|x-\xi|^{s}, |z|^{s}\} \min\{1, |z|^{\alpha s}\}|z|^{-d-s}\,dz \Bigr)\\
&\leq C\|\omega\|_{\mathcal{C}^{s}_x((-5, 0]\times\mathbb{R}^d)} |x-\xi|^{\alpha s}(1+|\ln |x-\xi||)\\
&\leq C\|\omega\|_{\mathcal{C}^{s}_x((-5, 0]\times\mathbb{R}^d)}|x-\xi|^{\beta s}.
\end{align*}
Moreover, for any $0<|t-\tau|\leq \frac{1}{e}$  we have
\begin{align*}
|Q_\omega (t, x)-Q_\omega (\tau, x)|
&\leq \left|\int_{\mathbb{R}^d}\bigl( \omega(t, x+z)-\omega(t, x)\bigr) \bigl(Q(t, x, z)-Q(\tau, x, z)\bigr)\,dz\right|\\
&\qquad+\left|\int_{\mathbb{R}^d}\bigl( \omega(t, x+z)-\omega(t, x)+\omega(\tau, x+z)-\omega(\tau, x)\bigr) Q(\tau, x, z)\,dz\right|\\
&\leq 2\|\omega\|_{\mathcal{C}^{s}_x((-5, 0]\times\mathbb{R}^d)}\int_{\mathbb{R}^d}\min\{1, |z|^{s}\} |Q(t, x, z)-Q(\tau, x, z)|\,dz\notag\\
&\qquad+\|\omega\|_{\mathcal{C}^{1, s}_{t, x}((-5, 0]\times\mathbb{R}^d)}\int_{\mathbb{R}^d}\min\{|t-\tau|, |z|^{s}\} |Q(\tau, x, z)|\,dz\notag\\
&\leq C\|\omega\|_{\mathcal{C}^{1, s}_{t, x}((-5, 0]\times\mathbb{R}^d)}\Bigl(\int_{\mathbb{R}^d}\min\{1, |z|^{s}\} \min\{|t-\tau|^\alpha, |z|^{\alpha s}\} |z|^{-d-s}\,dz\\
&\qquad+\int_{\mathbb{R}^d} \min\{|t-\tau|, |z|^{s}\} \min\{1, |z|^{\alpha s}\}|z|^{-d-s}\,dz \Bigr)\\
&\leq C\|\omega\|_{\mathcal{C}^{1,  s}_{t, x}((-5, 0]\times\mathbb{R}^d)} |t-\tau|^{\alpha}(1+|\ln |t-\tau||)\\
&\leq C\|\omega\|_{\mathcal{C}^{1,  s}_{t, x}((-5, 0]\times\mathbb{R}^d)} |t-\tau|^{\beta}.
\end{align*}
Thus we conclude that for $\beta<\alpha$, one has indeed
$\|Q_\omega\|_{\mathcal{C}^{\beta, \beta s}((-5, 0]\times\mathbb{R}^d)}\leq C\|\omega\|_{\mathcal{C}^{1, s}((-5, 0]\times\mathbb{R}^d)}.$
\end{proof}
We shall also need the following iteration lemma.
\begin{lemma}\label{L-A3}(Lemma 1.1 in \cite{GG82})
Let $h:[T_0, T_1]\to \mathbb{R}$ be nonnegative and bounded. Suppose that for all $0\leq T_0\leq t<\tau\leq T_1$ we have
\[
h(t)\leq A(\tau-t)^{-\gamma}+\frac{1}{2}h(\tau)
\]
with $\gamma>0$ and $A>0.$ Then there exists $C=C(\gamma)$ such that for all $T_0\leq t<\tau\leq T_1$ we have
\[
h(t)\leq CA(\tau-t)^{-\gamma}.
\]
\end{lemma}

We are ready to prove Theorem \ref{Th-Schauder-estimate}.
\begin{proof}
Let us now consider $\omega\in\mathcal{C}^{1+\alpha, (1+\alpha)s}((-6, 0]\times\mathbb{R}^d)$ a solution of the full equation~\eqref{schauder-eq}:
 \[
\partial_t \omega=\int_{\mathbb{R}^d} \bigl(\omega(t, x+y)-\omega(t, x)\bigr)\, L(t, x, y)\,dy
+\int_{\mathbb{R}^d} \bigl(\omega(t, x+y)-\omega(t, x)\bigr) \,Q(t, x, y)\,dy+ \phi(t, x).
\]
By Proposition \ref{P-A1} and Proposition \ref{P-A2}, we have for $\beta\in(0, \alpha),$
\begin{align*}
\|\omega\|_{\mathcal{C}^{1+\beta, (1+\beta)s}((-2, 0]\times\mathbb{R}^d)}&\leq C(\|\omega\|_{L^\infty((-5, 0]\times\mathbb{R}^d)}+\|Q_{\omega}, \phi\|_{\mathcal{C}^{\beta, \beta s}((-5, 0]\times\mathbb{R}^d)})\\
&\leq C(\|\omega\|_{\mathcal{C}^{1, s}((-5, 0]\times\mathbb{R}^d)}+\|\phi\|_{\mathcal{C}^{\beta, \beta s}((-5, 0]\times\mathbb{R}^d)}).
\end{align*}

\smallskip
We first handle the case $(1+\beta)s \leq1.$
Define (using the notation~\eqref{holder_seminorm} for H\"older's semi-norms):
$$h_\omega(\gamma, \tau):=\left\{
\begin{array}{ll}
[\partial_t \omega]_{\mathcal{C}^{\gamma, (1+\gamma)s}((\tau, 0]\times\mathbb{R}^d)}+[\omega]_{\mathcal{C}^{\gamma, (1+\gamma)s}((\tau, 0]\times\mathbb{R}^d)} &{\rm{if} ~~\gamma\in(0, \beta],}\\
\|[\partial_t \omega]_{\mathcal{C}^{s}_x(\mathbb{R}^d)}\|_{L^\infty_t((\tau, 0])}+\|[\omega]_{\mathcal{C}^{s}_x(\mathbb{R}^d)}\|_{L^\infty_t((\tau, 0])} &{\rm{if} ~~\gamma=0}.
\end{array}
\right.
$$
We just proved that
\begin{align}
h_\omega(\beta, -2)\leq C(\|\omega\|_{L^{\infty}((-5, 0]\times\mathbb{R}^d)}+
\|\phi\|_{\mathcal{C}^{\beta, \beta s}((-5, 0]\times\mathbb{R}^d)}+h_\omega(0, -5)).\label{iteration-eq-h}
\end{align}
For every $-2<\tau_0<\tau\leq -1,$ we let
$$\psi(t, x):=\omega(\mu t+t_*, \mu x)\quad {\rm{with}} \quad \mu:=\frac{\tau-\tau_0}{3},~~t_*:=\frac{5\tau-2\tau_0}{3},$$
then $\psi(t, x)$ satisfies that
\begin{align*}
\partial_t \psi(t, x)&=\int_{\mathbb{R}^d} \bigl(\psi(t, x+y)-\psi(t, x)\bigr)\, \widetilde L(t, x, y)\,dy\\
&\quad\quad+\int_{\mathbb{R}^d} \bigl(\psi(t, x+y)-\psi(t, x)\bigr) \,\widetilde Q(t, x, y)\,dy+ \widetilde \phi(t, x) \quad{\rm{in}}(-6, 0]\times\mathbb{R}^d,
\end{align*}
with
$$\widetilde L(t, x, y)=\mu^{d+1} L(\mu t+t_*, \mu x, \mu y),
\quad
\widetilde Q(t, x, y)=\mu^{d+1}Q(\mu t+t_*, \mu x, \mu y)$$
and
$\widetilde\phi(t, x)=\phi(\mu t+t_*, \mu x).$ As $\mu<1,$ and each of $\widetilde L, \widetilde Q, \widetilde\phi$ satisfies the same assumptions on $L, Q, \phi,$ respectively, the estimate \eqref{iteration-eq-h} holds true for $\psi$ as well. Noticing that
$$h_{\psi}(\beta, -2)\geq (\mu+1)\min\{\mu^{\beta}, \mu^{(\beta+1)s}\} h_{\omega}(\beta, \tau)\geq \mu^{\beta+s}h_{\omega}(\beta, \tau)$$
and
$$h_{\psi}(0, -5)\leq (\mu+1)\mu^{s} h_{\omega}(\beta, \tau_0)\leq 2\mu^{s} h_{\omega}(0, \tau_0),$$
we have 
$$h_\omega (\beta, \tau)\leq \frac{C}{|\tau-\tau_0|^{\beta+s}} \left(\|\omega\|_{L^{\infty}((-5, 0]\times\mathbb{R}^d)}+\|\phi\|_{\mathcal{C}^{\beta, \beta s}((-5, 0]\times\mathbb{R}^d)}\right)+\frac{C}{|\tau-\tau_0|^\beta}h_{\omega}(0, \tau_0).$$
By  interpolation inequality in H\"older spaces and Young's inequality, we know that for every $\epsilon_0<1,$ there exists $C>0$ independent of $\epsilon_0$ such that
\begin{align*}
h_\omega(0, \tau_0)
&\leq [\partial_t \omega]_{\mathcal{C}^{\frac{\beta}{(1+\beta)s+1}, s}((\tau_0, 0]\times\mathbb{R}^d)}
+ [\omega]_{\mathcal{C}^{\frac{\beta}{(1+\beta)s+1}, s}((\tau_0, 0]\times\mathbb{R}^d)}\\
&\leq [\partial_t \omega]_{\mathcal{C}^{\beta, (1+\beta)s}}^{\frac{1+s}{(1+\beta)s+1}}\,\,\|\omega\|_{L^\infty}^{1-\frac{1+s}{(1+\beta)s+1}}
+ [\omega]_{\mathcal{C}^{\beta, (1+\beta)s}}^{\frac{1+s}{(1+\beta)s+1}}\,\,\|\omega\|_{L^\infty}^{1-\frac{1+s}{(1+\beta)s+1}}\\
&\leq \epsilon_0 ( [\partial_t \omega]_{\mathcal{C}^{\beta, (1+\beta)s}}+ [ \omega]_{\mathcal{C}^{\beta, (1+\beta)s}})+C\epsilon_0^{-\frac{1+s}{\beta s}}\|\omega\|_{L^\infty}\\
&= \epsilon_0 h_\omega(\beta, \tau_0)+C\epsilon_0^{-\frac{1+s}{\beta s}}\|\omega\|_{L^\infty((\tau_0, 0]\times\mathbb{R}^d)}.
\end{align*}
Choosing $\epsilon_0=\frac{|\tau-\tau_0|^\beta}{2C},$ we get 
\begin{align*}
h_\omega (\beta, \tau)\leq \frac{1}{2}h_{\omega}(\beta, \tau_0)+\frac{C}{|\tau-\tau_0|^{\beta+1+1/s}}(\|\omega\|_{L^{\infty}((-5, 0]\times\mathbb{R}^d)}+\|\phi\|_{\mathcal{C}^{\beta, \beta s}((-5, 0]\times\mathbb{R}^d)}).
\end{align*}
Thanks to Lemma \ref{L-A3}, we thus have that
$$h_\omega(\beta, -1)\leq C(\|\omega\|_{L^{\infty}((-1, 0]\times\mathbb{R}^d)}+\|\phi\|_{\mathcal{C}^{\beta, \beta s}((-1, 0]\times\mathbb{R}^d)}).$$

\smallskip
For the case $(1+\beta)s> 1,$ we alter the definition
$$h_\omega(\gamma, \tau)=\left\{
\begin{array}{ll}
[\partial_t \omega]_{\mathcal{C}^{\gamma, (1+\gamma)s-1}((\tau, 0]\times\mathbb{R}^d)}+[\nabla_x \omega]_{\mathcal{C}^{\gamma, (1+\gamma)s-1}((\tau, 0]\times\mathbb{R}^d)} & {\rm{if} ~~\gamma\in(0, \beta]}\\
\|\partial_t \omega\|_{L^\infty((\tau, 0]\times\mathbb{R}^d)}+\|\nabla_x \omega\|_{L^\infty_t((\tau, 0]\times\mathbb{R}^d)}
&{\rm{if} ~~\gamma=0},
\end{array}
\right.
$$
as  in \cite{Im2} for the case $s=1$. 
Then by Proposition \ref{P-A2}, we have
\begin{align*}
\|\omega\|_{\mathcal{C}^{1+\beta, (1+\beta)s}((-2, 0]\times\mathbb{R}^d)}&\leq C(\|\omega\|_{L^\infty((-5, 0]\times\mathbb{R}^d)}+\|Q_{\omega}, \phi\|_{\mathcal{C}^{\beta, \beta s}((-5, 0]\times\mathbb{R}^d)})\\
&\leq C(\|\omega\|_{\mathcal{C}^{1, s}((-5, 0]\times\mathbb{R}^d)}+\|\phi\|_{\mathcal{C}^{\beta, \beta s}((-5, 0]\times\mathbb{R}^d)}).
\end{align*}
This shows that \eqref{iteration-eq-h} is satisfied. The rest of the proof is then similar to the previous case.
\end{proof}

\bigskip
{\small
\noindent\textit{Acknowledgments.}
 The first author is partly funded by the B\'{e}zout Labex through  ANR, reference \textsc{ANR-10-LABX-58}.
 Both authors thank the ANR, project  \textsc{ANR-15-CE40-001, INFAMIE} for offering partial support.}

\bigbreak\vspace*{3em}
\noindent$^\ast$ \textbf{Jin Tan}\\
\textsc{Univ Paris-Est Creteil,  CNRS, LAMA, F-94010 Creteil,  France \newline
\& Laboratoire de mathématiques AGM, UMR 8088 CNRS, Cergy Paris Université}\\
E-mail address: jin.tan@u-pec.fr
\bigskip\medbreak
\noindent$^\dagger$ \textbf{Francois Vigneron}\\
\textsc{
Universit\'{e} de Reims Champagne-Ardenne (URCA),
Laboratoire de Math\'{e}matiques de Reims (LMR),
UMR 9008 CNRS, 
Moulin de la Housse, BP 1039, F-51687 Reims Cedex 2, France}\\
E-mail address: francois.vigneron@univ-reims.fr

\end{document}